\documentclass[article,12pt]{amsart}
\usepackage[top=25mm,bottom=25mm,left=25mm,right=25mm]{geometry}
\usepackage{graphicx}

\newtheorem{lemma}{Lemma}

\newtheorem{theorem}{Theorem}
\newtheorem{corollary}{Corollary}

\newtheorem{example}{Example}

\DeclareMathOperator{\dist}{dist}
\DeclareMathOperator{\loc}{loc}

\newcommand{\cA}{\mathcal{A}}
\newcommand{\cL}{\mathcal{L}}
\newcommand{\cJ}{\mathcal{J}}
\newcommand{\R}{\mathbf{R}}

\newcommand{\pr}{\mathbf{P}}
\newcommand{\ex}{\mathbf{E}}
\newcommand{\Rd}{\mathbf{R}^d}
\newcommand{\N}{\mathbf{N}}

\begin{document}
\title[IU for Schr{\"o}dinger operators based on fractional Laplacians]{Intrinsic ultracontractivity for Schr{\"o}dinger operators based on fractional Laplacians}
\author{ Kamil Kaleta, Tadeusz Kulczycki}
\address{Kamil Kaleta, Institute of Mathematics and Computer Science \\ Wroc{\l}aw University of Technology 
\\ Wyb. Wyspia{\'n}skiego 27, 50-370 Wroc{\l}aw, Poland}
\email{kamil.kaleta@pwr.wroc.pl}
\address{Tadeusz Kulczycki, Institute of Mathematics, Polish Academy of Sciences, ul. Kopernika 18, 51-617 Wroc{\l}aw, Poland.  Institute of Mathematics and Computer Science, Wroc{\l}aw University of Technology, Wyb. Wyspia{\'n}skiego 27, 50-370 Wroc{\l}aw, Poland}
\email{tkulczycki@impan.pan.wroc.pl}
\thanks{The authors were partially supported by KBN grant.}

\begin{abstract}
We study the Feynman-Kac semigroup generated by the Schr{\"o}dinger operator based on the fractional Laplacian $-(-\Delta)^{\alpha/2} - q$ in $\Rd$, for $q \ge 0$, $\alpha \in (0,2)$. We obtain sharp estimates of the first eigenfunction $\varphi_1$ of the Schr{\"o}dinger operator and conditions equivalent to intrinsic ultracontractivity of the Feynman-Kac semigroup. For potentials $q$ such that $\lim_{|x| \to \infty } q(x) = \infty$ and comparable on unit balls we obtain that $\varphi_1(x)$ is comparable to $(|x| + 1)^{-d - \alpha} (q(x) + 1)^{-1}$ and intrinsic ultracontractivity holds iff $\lim_{|x| \to \infty} q(x)/\log|x| = \infty$. Proofs are based on uniform estimates of $q$-harmonic functions.
\end{abstract}

\maketitle

\section{Introduction and statement of results}

The aim of this paper is to study intrinsic ultracontractivity  for Feynman-Kac semigroups generated by Schr{\"o}dinger operators based on fractional Laplacians and obtain sharp estimates of the first eigenfunction of these operators. Mainly we use probabilistic methods.

Let $X_t$ be a symmetric $\alpha$-stable process in $\Rd$, $d \in \N$, $\alpha \in (0,2)$. This process is a Markov process with independent and homogeneous increments and the characteristic function of  the form $\ex^0(\exp(i \xi X_t)) = \exp(-t |\xi|^{\alpha})$, $\xi \in \Rd$, $t > 0$. As usual $\ex^x$, $x \in \Rd$ denotes the expected value of the process starting from $x \in \Rd$. 

The Feynman-Kac semigroup $(T_t)$, $t > 0$ for $X_t$ and a locally bounded, measuarable potential $0 \le q(x) < \infty$ is defined as follows
\begin{align}
\label{def:FKS}
T_t f(x) & = \ex^x\left[\exp\left(-\int_0^t q(X_s) \, ds\right)f(X_t)\right]
\, , & f \in L^2(\R^d), x \in \R^d
\, .
\end{align}
The generator of this semigroup is the Schr{\"o}dinger operator based on fractional Laplacian
$$
-(-\Delta)^{\alpha/2} - q.
$$

In recent years Schr{\"o}dinger operators based on non-local pseudodifferential operators have been intensively studied. 
For example in 2008 R. Frank, E. Lieb and R. Seiringer \cite{bib:FLS1} showed Hardy-Lieb-Thirring inequality for such Schr{\"o}dinger operators. This was done in connections with the problem of the stability of relativistic matter, which problem is closely related to non-local Schr{\"o}dinger operators and has been widely studied see e.g. \cite{bib:FLS2, bib:FL, bib:LY, bib:LSS}.  In the last 20 years there were obtained many results for Schr{\"o}dinger operators based on fractional Laplacians \cite{bib:C, bib:CMS, bib:Z, bib:CS1, bib:CS2, bib:Ku2, bib:BB1, bib:BB2, bib:CS}. These results concern the conditional gauge theorem, $q$-harmonic functions, intrinsic ultracontractivity, estimates of eigenfunctions. Most of these results are obtained for Schr{\"o}dinger operators on bounded domains and not on the whole $\Rd$ as in our paper.

The paper which is the most related to our paper is \cite{bib:KS}, where similar problems were studied for the Schr{\"o}dinger operator $-((-\Delta + m^{2/\alpha})^{\alpha/2} - m) - q$, where $m > 0$. The operator $-((-\Delta + m^{2/\alpha})^{\alpha/2} - m)$ for $m > 0$ is an infinitesimal generator of the relativistic $\alpha$-stable process \cite{bib:R}. It is worth to point out that there are huge differences between our paper and \cite{bib:KS}. Our paper not only concerns different Schr{\"o}dinger operators $-(-\Delta)^{\alpha/2} - q$ but uses completely new methods. These methods may be described as the use of uniform estimates of $q$-harmonic functions in proving intrinsic ultracontractivity. We take these methods from M. Kwa{\'s}nicki paper \cite{bib:Kw}, where he used uniform boundary Harnack principle (uBHP) for $\alpha$-harmonic functions (shown in \cite{bib:BKK}) in proving intrinsic ultracontractivity. It is worth to point out that both the proof of uBHP in \cite{bib:BKK} and our uniform estimates of $q$-harmonic functions (Lemma \ref{lm:est}, Theorem \ref{th:bhi}, Corollary \ref{cor:bhisc}) use a very important idea from R. Song and J. M. Wu  paper \cite[proof of Lemma 3.3]{bib:SW}. Let us also note that the results proven in our paper are much sharper than those in \cite{bib:KS}. In particular we obtain characterization of intrinsic ultracontractivity and sharp estimates of the first eigenfunction (Theorem \ref{th:eig}, Theorem \ref{th:char}) for much wider class of potentials $q$ than in Theorem 1.6 in \cite{bib:KS}. This gives e.g. a very natural property of intrinsic ultracontractivity stated in Corollary \ref{cor:order}. There is no such result in \cite{bib:KS}. 

Now we introduce notation needed in formulating our results. The Feynman-Kac semigroup $(T_t)$ is given by the kernel $u(t,x,y)$, that is 
$$
T_t f(x) = \int_{\Rd} u(t,x,y) f(y) \, dy, \quad x \in \Rd, \quad f \in L^2(\Rd).
$$
For each $t > 0$ the kernel $u(t,\cdot,\cdot)$ is continuous and bounded on $\Rd \times \Rd$. For any $t > 0$, $x,y \in \Rd$ the kernel is strictly positive. The proof of these properties is standard. It is similar to the proofs for the classical Feynman-Kac semigroups (see e.g. \cite{bib:CZ}). For the convenience of the reader we write the short proof of properties of $u(t,x,y)$ in Lemma \ref{lm:kernel}.

Our first result gives a simple criterion of the compactness of operators $T_t$. By $L_{\loc}^{\infty}$ we denote the class of locally bounded, measurable functions $q: \Rd \to \R$.

\begin{lemma}
\label{lm:compact}
Let $q \in L_{\loc}^{\infty}$, $q \ge 0$. If $q(x) \to \infty$ as $|x| \to \infty$ then for all $t > 0$ operators $T_t$ are compact.
\end{lemma}

On the other hand, if there is an infinite sequence of disjoint unit balls such that $q$ is bounded on this sequence, then $T_t$ are not compact (for justification of this statement see the proof of Theorem 1.1 in \cite{bib:KS}, page 5039).

When for all $t > 0$ operators $T_t$ are compact, then the general theory of semigroups (see e.g. \cite{bib:D}) gives the following standard results. There exists an orthonormal basis in $L^2(\Rd)$ consisting of eigenfunctions $\varphi_n$ such that $T_t \varphi_n = e^{-\lambda_n t} \varphi_n$, where $0 < \lambda_1 < \lambda_2 \le \lambda_3 \le \ldots \to \infty$. All $\varphi_n$ are continuous and bounded. The first eigenfunction $\varphi_1$ can be assumed to be strictly positive. 

Let us assume that for all $t > 0$ operators $T_t$ are compact. The semigroup $(T_t)$ is called {\it{intrinsically ultracontractive}} (abbreviated as IU) if for each $t > 0$ there is a constant $C_{q,t}$ such that 
\begin{equation}
\label{def:IU}
u(t,x,y) \le C_{q,t} \, \varphi_1(x) \varphi_1(y), \quad x,y \in \Rd.
\end{equation}

The notion of IU was introduced in \cite{bib:DS} for very general semigroups. Important examples of such semigroups are the semigroups of elliptic operators $H_0$ and the semigroups of Schr{\"o}dinger operators $H = H_0 - q$ both on $\Rd$, as well on domains $D \subset \Rd$ with Dirichlet boundary conditions. IU for such semigroups has been widely studied, see e.g. \cite{bib:B, bib:Da, bib:D, bib:BD}. IU has also been studied for semigroups generated by $-(-\Delta)^{\alpha/2}$ and $-(-\Delta)^{\alpha/2} - q$ on bounded domains \cite{bib:CS1, bib:CS2, bib:Ku2}.

The classical result for the Feynman Kac semigroup $(T_t)$ on $\Rd$ generated by $H = \Delta - q$ is the following fact (Corollary 4.5.5, Theorem 4.5.11 and Corollary 4.5.8 in \cite{bib:D}, cf. also \cite{bib:DS}). If $q(x) = |x|^{\beta}$, then $(T_t)$ is IU iff $\beta > 2$. Moreover for $\beta > 2$ we have $c f(x) \le \varphi_1(x) \le C f(x)$, $|x| > 1$, where $f(x) = |x|^{-\beta/4 + (d - 1)/2} \exp(-2 |x|^{1 + \beta/2} /(2+\beta))$.

Now we come to formulating main results of our paper. The Feynman-Kac functional is defined as $e_q(t)=\exp(-\int_0^t q(X_s)ds)$, $t > 0$. For $q \in L_{\loc}^{\infty}$, $q \ge 0$ and an open set $D \subset \R^d$ and $x \in D$ we  define 
\begin{align*}
v_D(x) = \ex^x \left[\int_0^{\tau_D} e_q(t) dt\right]
\, ,
\end{align*}
where $\tau_D = \inf\{t > 0: \, X_t \notin D\}$ is the first exit time from $D$. For a regular (say bounded Lipschitz) open set $D$ we have $v_D(x) = \int_D V_D(x,y)dy$, where $V_D(x,y)$  is a {\it{$q$-Green function}} of $D$ (for a definition of $V_D(x,y)$ see Preliminaries).

The next theorem gives sharp estimates of the first eigenfunction. 

\begin{theorem}
\label{th:eig}
Let $q \in L^\infty_{\loc}$, $q \geq 0$ and $q(x) \rightarrow \infty$ as $|x| \rightarrow \infty$. Then there exist constants $C_q^{(1)}$ and $C_q^{(2)}$ such that for all $x \in \R^d$ and $D=B(x,1)$
\begin{align}
\label{eq:eig0}
  \frac{C_q^{(1)} v_D(x)}{(1+|x|)^{d+\alpha}}
  & \le
   \, \varphi_1(x)
   \le
    \frac{C_q^{(2)} v_D(x)}{(1+|x|)^{d+\alpha}}
  \, .
\end{align}
Additionally, $v_D(x)$ can be replaced by $\int_{\R^d}V(x,y)dy$, where $V(x,y) = \int_0^{\infty} u(t,x,y) \, dt$.
\end{theorem}

An essential dependence between estimates of the first eigenfuncton and IU already comes out in the classical setting. In our case a knowledge of asymptotic behaviour of the first eigenfunction also leads us to obtain criteria for IU.

\begin{theorem}
\label{th:char}
Let $q \in L^\infty_{\loc}$, $q \geq 0$ and $q(x) \rightarrow \infty$ as $|x| \rightarrow \infty$.
The following conditions are equivalent:
\begin{itemize}
	\item[(i)] The semigroup $(T_t)$ is intrinsically ultracontractive.
	\item[(ii)] For any $t > 0$ there is a constant $C_{q,t}$ such that for all $x,y \in \R^d$ we have \\ $u(t,x,y) \leq C_{q,t} (1+|x|)^{-d-\alpha} (1+|y|)^{-d-\alpha}$.
	\item[(iii)] For any $t > 0$ there is a constant $C_{q,t}$ such that for all $r > 0$, $x \in \overline{B}(0,r)^c$ we have \\ $\ex^x[t<\tau_{\overline{B}(0,r)^c}; e_q(t)] \leq      
	             C_{q,t} (1+r)^{-d-\alpha}$.
	\item[(iv)]  For any $t > 0$ there is a constant $C_{q,t}$ such that for all $x \in \R^d$ we have \\ $T_t \chi_{\R^d}(x) \leq C_{q,t} (1+|x|)^{-d-\alpha}$.
\end{itemize}
\end{theorem}

The next corollaries follow immediately from equivalence of conditions (i),(ii) and (i),(iii).

\begin{corollary}
Let $q \in L^\infty_{\loc}$, $q \geq 0$ and $q(x) \rightarrow \infty$ as $|x| \rightarrow \infty$. If the semigroup $(T_t)$ is intrinsically ultracontractive, then each $T_t$ is a Hilbert-Schmidt operator.
\end{corollary}

\begin{corollary}
\label{cor:order}
Let $q_1, q_2 \in L^\infty_{\loc}$, $q_1 \geq 0$ and $q_1(x) \rightarrow \infty$ as $|x| \rightarrow \infty$. If the semigroup $(T_t)$ for potential $q_1$ is intrinsically ultracontractive and $q_1 \leq q_2$, then $(T_t)$ for potential $q_2$ is intrinsically ultracontractive.
\end{corollary}

A simple consequence of Theorem \ref{th:char} is the sufficient condition for IU, which can be formulated in terms of the behaviour of the potential $q$ at infinity.

\begin{theorem}
\label{th:suff}
Let $q \in L^\infty_{\loc}$, $q \geq 0$. If $\lim_{|x| \rightarrow \infty}\frac{q(x)}{\log |x|} = \infty$, then the operators $T_t$ are compact and the semigroup $(T_t)$ is intrinsically ultracontractive.
\end{theorem}

A neccesary condition for IU can be stated as follows. 

\begin{theorem}
\label{th:nec}
Let $q \in L^\infty_{\loc}$, $q \geq 0$ and $q(x) \rightarrow \infty$ as $|x| \rightarrow \infty$. If the semigroup $(T_t)$ is intrinsically ultracontractive, then for any $\epsilon \in (0,1]$ we have $\lim_{|x| \rightarrow \infty}\frac{\sup_{y \in B(x,\epsilon)} q(y)}{\log |x|} = \infty$.
\end{theorem}

The next theorem, arising from Theorem \ref{th:eig}, contains  more explicit estimates for the first eigenfunction.

\begin{theorem}
\label{th:eig1}
Let $q \in L^\infty_{\loc}$, $q \geq 0$ and $q(x) \rightarrow \infty$ as $|x| \rightarrow \infty$. Let $x \in \R^d$ and let $M_{q,x}\geq1$ be the constant such that
\begin{align*}
 M_{q,x}^{-1}(1+q(x)) & \le q(y) \le M_{q,x} (1+q(x))
  \, , & y \in B(x,1)
  \, .
\end{align*}
Then we have the following estimates
\begin{align}
\label{eq:eig1}
  \frac{C_{q,x}^{(1)}}{(1+q(x))(1+|x|)^{d+\alpha}}
  & \le
   \, \varphi_1(x)
   \le
    \frac{C_{q,x}^{(2)}}{(1+q(x))(1+|x|)^{d+\alpha}}
  \, ,
\end{align}
with constants $C_{q,x}^{(1)}=2^{-1}C_q^{(1)}M_{q,x}^{-1}\pr^0(\tau_{B(0,1)}>1)$ and $C_{q,x}^{(2)}=C_q^{(2)}M_{q,x}$, where $C_q^{(1)}, C_q^{(2)}$ are the constants from \eqref{eq:eig0}.
\end{theorem}

A natural conclusion from the above theorem is the following  result for potentials $q$ comparable on unit balls.

\begin{corollary}
\label{cor:eig2}
Let $q \in L^\infty_{\loc}$, $q \geq 0$ and $q(x) \rightarrow \infty$ as $|x| \rightarrow \infty$. Let $M_q \geq1$ be a uniform constant such that 
\begin{align}
\label{eq:def2}
 M_q^{-1}(1+q(x)) & \le q(y) \le M_q (1+q(x))
  \, , & x \in \R^d, y \in B(x,1)
  \, .
\end{align}
Then, for all $x \in \R^d$, we have
\begin{align}
\label{eq:eig2}
  \frac{C_q^{(3)}}{(1+q(x))(1+|x|)^{d+\alpha}}
  & \le
   \, \varphi_1(x)
   \le
    \frac{C_q^{(4)}}{(1+q(x))(1+|x|)^{d+\alpha}}
  \, .
\end{align}
\end{corollary}

Examples of $q$ satisfying \eqref{eq:def2} are $q(x) = |x|^{\beta}$, $q(x) = \exp(\beta |x|)$, $\beta >0$ but not $q(x) = \exp(|x|^2)$. The following example shows that the assumption \eqref{eq:def2} in the Corollary \ref{cor:eig2} is essential.

\begin{example}
\label{ex:ex2}
Let $2^\alpha < a_1 < a_2 < a_3 < ... \rightarrow \infty$ be a sequence such that $\lim_{n \rightarrow \infty} \frac{a_{n+1}}{a_n}=\infty$. Set $r_n = \frac{1}{a_n^{1/\alpha}}$. Define:
\begin{equation*}
\label{def}
q(x) = 
\begin{cases}
a_1				 		                                                            &	\text{for} \quad |x| \leq r_1,                                 \\
a_n 																																			& \text{for} \quad n-1+r_n \leq |x| \leq n-r_{n+1}, n \geq 1,    \\	
\left(\frac{a_{n+1}-a_n}{2r_{n+1}}\right) (|x|-n+r_{n+1}) + a_n           &	\text{for} \quad n-r_{n+1} \leq |x| \leq n+r_{n+1}, n \geq 1.  \\		
\end{cases}
\end{equation*}
The potential $q$ is a nonnegative, locally bounded and continuous function such that $q(x) \rightarrow \infty$ as $|x| \rightarrow \infty$. However, the upper bound estimate in \eqref{eq:eig2} does not hold.
\end{example}

The justification of this example will be given in the last section. The justification is based on the estimates of the heat kernel for Dirichlet fractional Laplacian obtained by Z.-Q. Chen, P. Kim, R. Song in \cite[Theorem 1.1]{bib:CKS} and results of K. Bogdan, T. Grzywny \cite[Corollary 1]{bib:BG}.

The next corollary follows from Theorem \ref{th:suff} and Theorem \ref{th:nec} and gives the condition equivalent to IU for potentials comparable on unit balls.

\begin{corollary}
\label{cor:iuunit}
Let $q \in L^\infty_{\loc}$, $q \geq 0$ and $q(x) \rightarrow \infty$ as $|x| \rightarrow \infty$. If the condition \eqref{eq:def2} is satisfied, then the semigroup $(T_t)$ is intrinsically ultracontractive if and only if $\frac{q(x)}{\log |x|} \rightarrow \infty$ as $|x| \rightarrow \infty$.
\end{corollary}

The paper is organized as follows. In Preliminaries section we introduce notation and collect various facts which are needed in the sequel. In Section 3 we prove uniform  estimates of $q$-harmonic functions: Lemma \ref{lm:est}, Theorem \ref{th:bhi}, Corollary \ref{cor:bhisc} ("uniform" means not depending on the potential $q$). These results may be of independent interest. In section 4 we study conditions for compactness of $T_t$. Section 5 contains the proofs of the first eigenfunction estimates and the proofs of main theorems concerning intrinsic ultracontractivity. Proofs of more exact results for potentials comparable on unit balls are contained in the last section.

\section{Preliminaries}

Let $\alpha \in (0,2)$. For $x \in \R^d$ and a set $U \subset \R^d$, the symbols $|x|$, $|U|$ denote the Euclidean norm of $x$ and the $d$-dimensional Lebesque measure of the set $U$. By $B(x,r)$, $x \in \R^d$, $r >0$, we denote the standard Euclidean ball. The set $U^c$ is a complement of an arbitrary subset $U \subset \R^d$ and $\partial U$ denotes its boundary. For $x \in U$ let $\delta_U(x) = \dist(x,\partial U) = \inf\left\{|x-y|: y \in \partial U \right\}$. For a set $U$ and $r > 0$ we also define $rU = \left\{rx: x \in U\right\}$.

By $C_\kappa$ we always mean a strictly positive and finite constant depending on $\alpha$, $d$ and parameter $\kappa$ (we always omit dependence on $\alpha$ and $d$, and do not indicate it). We adapt the convention that constants may change their values from one use to the next. Sometimes we will write $C_\kappa^{(1)}, C_\kappa^{(2)}$ when we need to refer to concrete constants in the sequel.

Now we briefly introduce the needed properties of the process $X_t$ and some facts from its potential theory. The reader can find the wider introduction to the potential theory of stable processes in \cite{bib:Bo, bib:Ku1, bib:CS1}. $X_t$ is a standard rotation invariant $\alpha$-stable L\'evy process (i.e. homogenous, with independent increments) with L\'evy measure given by the density $\nu(x) = \cA |x|^{-d-\alpha}$, where $\cA=2^\alpha \pi^{-d/2} \Gamma((d+\alpha)/2)|\Gamma(-\alpha/2)|^{-1}$. By $\pr^x$ we denote the distribution of the process starting from $x \in \Rd$. For each fixed $t>0$ the transition density $p(t,y-x)$ of the process $X_t$ starting from $x \in \Rd$ is a continuous and bounded function on $\R^d \times \R^d$ satisfying the following estimates
\begin{align}
\label{eq:weaksc}
C^{-1} \min \left\{\frac{t}{|y-x|^{d+\alpha}}, t^{-d/\alpha} \right\} & \leq p(t,y-x) \leq C \min\left\{\frac{t}{|y-x|^{d+\alpha}}, t^{-d/\alpha} \right\}
\, , & x,y \in \R^d
\, .
\end{align}

We denote $P_t f(x) = \ex^x f(X_t) = \int_{\R^d} f(y) p(t,y-x) dy$. Using estimates \eqref{eq:weaksc}, we can simply show that operators $P_t:L^1(\R^d) \rightarrow L^\infty(\R^d)$, $P_t:L^1(\R^d) \rightarrow L^1(\R^d)$, $P_t:L^\infty(\R^d) \rightarrow L^\infty(\R^d)$ are bounded. These properties will be crucial in the proof of Lemma \ref{lm:kernel}.

By $p_D(t,x,y)$ we denote the transition density of the process killed on exiting an open set $D$. We have
\begin{align*}
p_D(t,x,y) & = p(t,y-x) - \ex^x[\tau_D \leq t; p(t-\tau_D,y-X_{\tau_D})]
 \, , & x, y \in D, t > 0
\, .
\end{align*}
We put $p_D(t,x,y) = 0$ whenever $x \notin D$ or $y \notin D$. It is clear that $\pr^x(\tau_D > t) = \int_D p_D(t,x,y)dy$. 

A function $F: \R^d \rightarrow \R$ is called $C^{1,1}$ if it has a first derivative $F'$ and there exists a constant $\delta$ such that for all 
$x, y \in \R^d$ we have $|F'(x)- F'(y)| \leq \delta |x - y|$. We say that a bounded open set $D \subset \R^d$ is a $C^{1,1}$ domain if for each
$x \in \partial D$ there are: a $C^{1,1}$ function $F_x: \R^{d-1} \rightarrow \R$ (with a constant $\delta = \delta(D)$), an orthonormal coordinate 
system $CS_x$, and a constant $\eta = \eta(D)$ such that if $y = (y_1, ..., y_d)$ in $CS_x$ coordinates, then
\begin{align*}
D \cap B(x,\eta) = \left\{y:y_d > F_x(y_1, ..., y_{d-1})\right\} \cap B(x,\eta)
\, .
\end{align*}
It was proved in \cite[Theorem 1.1]{bib:CKS} that for $C^{1,1}$ domain $D$, $t \in (0,1]$, $x, y \in D$, we have 
\begin{align*}
C^{-1}\left(1 \wedge \frac{\delta_D^{\frac{\alpha}{2}}(x)}{\sqrt{t}}\right)
      \left(1 \wedge \frac{\delta_D^{\frac{\alpha}{2}}(y)}{\sqrt{t}}\right) p(t,y-x) 
& \leq p_D(t,x,y) \\
& \leq C\left(1 \wedge \frac{\delta_D^{\frac{\alpha}{2}}(x)}{\sqrt{t}}\right)
       \left(1 \wedge \frac{\delta_D^{\frac{\alpha}{2}}(y)}{\sqrt{t}}\right) p(t,y-x)
\, .
\end{align*}
The upper bound for semibounded convex domains was shown earlier, in \cite[Theorem 1.6]{bib:Si}. 

The following lemma was obtained in \cite[Corollary 1]{bib:BG} as a straightforward corollary from the above estimates of $p_D(t,x,y)$. It only will be used in the justification of Example \ref{ex:ex2}. 

\begin{lemma}
If $D \subset \R^d$ is a $C^{1,1}$ domain, then there is a constant $C$ such that 
\begin{align}
\label{eq:surv}
C^{-1}\left(1 \wedge \frac{\delta_D^{\frac{\alpha}{2}}(x)}{\sqrt{t}}\right) & \leq \pr^x(\tau_D > t) \leq C\left(1 \wedge \frac{\delta_D^{\frac{\alpha}{2}}(x)}{\sqrt{t}}\right)
 \, , & t \in (0,1], x \in D
\, .
\end{align}
\end{lemma}

The Green function of an open bounded set $D$ is defined by $G_D(x,y) = \int_0^\infty p_D(t,x,y)dt$. For nonnegative Borel function $f$ on $\R^d$ we have $\int_D G_D(x,y)f(y)dy = \ex^x\left[\int_0^{\tau_D} f(X_t) dt \right]$. In the sequel we will often use the following well known fact \cite{bib:G} $E^x(\tau_{B(0,r)}) = c (r^2 - |x|^2)^{\alpha/2}$, $r > 0$, $x \in B(0,r)$, $c = \Gamma(d/2)(2^{\alpha} \Gamma(1+\alpha/2)\Gamma((d+\alpha)/2))^{-1}$.

We now discuss properties of Feynman-Kac semigroups for Schr{\"o}dinger operators based on $-(-\Delta)^{\alpha/2}$. We refer the reader to \cite{bib:BB1, bib:BB2, bib:CS1} for more systematic treatment of Schr{\"o}dinger operators based on $-(-\Delta)^{\alpha/2}$.  

At first we prove the existence and basic properties of the kernel $u(t,x,y)$.

\begin{lemma}
\label{lm:kernel}
Let $q \in L^\infty_{\loc}$ and $q \geq 0$. We have:
\begin{itemize}
	\item[(i)]    $T_t f(x) \leq P_t f(x)$ for $f \geq 0$ on $\R^d$, $x \in \R^d$, $t>0$. 
	\item[(ii)]   For any $t > 0$, $T_t:L^\infty(\R^d) \rightarrow C_b(\R^d)$.
	\item[(iii)]  There exists a kernel $u(t,x,y)$ for $T_t$, i.e. $T_tf(x) = \int_{\R^d} u(t,x,y) f(y) dy$, $t>0$, $x \in \R^d$, $f \in L^p(\R^d) (1\leq                 p \leq \infty)$. For each fixed $t >0$, $u(t,x,y)$ is continuous and bounded on $\R^d \times \R^d$.
	\item[(iv)]   $u(t,x,y) = u(t,y,x)$, $t>0$, $x,y \in \R^d$.
	\item[(v)]    $0 < u(t,x,y) \leq p(t,y-x)$, $t > 0$, $x,y \in \R^d$.
\end{itemize}
\end{lemma}

The proof of this lemma is standard and is based on \cite[Section 3.2]{bib:CZ}. Similar arguments may be found in \cite[proof of Lemma 3.1]{bib:KS}. We repeat these arguments for the convenience of the reader. 

\begin{proof}
The property (i) is clear from definition of $T_t$ and our assumption that $q \geq 0$. 

For the proof of (ii) we put $q_n(x) = \chi_{B(0,n)}(x)q(x)$, $x \in \R^d$, $n=1,2,...$. By our assumption that $q \in L^\infty_{\loc}$, we have $q_n \in \cJ^\alpha$, $n=1,2,...$. $\cJ^\alpha$ is the Kato class, its definition may be found e.g. in (2.5) in \cite{bib:BB2}. For any $n$ we put $T_{t,n}f(x) = \ex^x[e_{q_n}(t)f(X_t)]$, $t>0$, $x \in \R^d$. By continuity and boundedness on $\R^d \times \R^d$ (for fixed $t>0$) of the density $p(t,y-x)$, we get $P_t:L^\infty(\R^d)\rightarrow C_b(\R^d)$. By this, formula (2.10) in \cite{bib:BB2} and the same argument as in the proofs of \cite[Propositions 3.11 and 3.12]{bib:CZ}, we also obtain that $T_{t,n}:L^\infty(\R^d)\rightarrow C_b(\R^d)$ for any $n=1,2...$. Furthermore,
\begin{align*}
|T_tf(x) - T_{t,n}f(x)| = |\ex^x[(e_q(t)-e_{q_n}(t))f(X_t)]| \leq \left\|f\right\|_\infty \pr^x(\tau_{B(0,n)} < t)
\, .
\end{align*}
Since for each fixed $t>0$ we have $\pr^x(\tau_{B(0,n)} < t) \rightarrow 0$ as $n \rightarrow \infty$, this implies (ii). 

Now we justify the properties (iii)-(v). From (i) and properties of $P_t$ we obtain that the operators $T_t:L^1(\R^d) \rightarrow L^\infty(\R^d)$ and $T_t:L^1(\R^d) \rightarrow L^1(\R^d)$ are bounded. By this and theorem of Dunford and Pettis \cite[Theorem A.1.1, Corollary A.1.2]{bib:S}(see also \cite{bib:CZ}), for each $t>0$, there exists a measurable on $\R^d \times \R^d$ kernel $u(t,x,y)$, $x,y \in \R^d$, for $T_t$, that is 
\begin{align*}
T_tf(x) & = \int_{\R^d} u(t,x,y)f(y)dy
\, , & f \in L^1(\R^d), t>0, x \in \R^d
\, . 
\end{align*} 
By (i) and properties of $P_t$, this representation also holds for all $f \in L^p(\R^d)$, $1 \leq p \leq \infty$.  The properties (i), definition of $T_t$ and the fact that $q \in L^\infty_{\loc}$ give that for each fixed $t>0$ and $x \in \R^d$ we have $0 < u(t,x,y) \leq p(t,y-x)$ for almost all $y \in \R^d$. We may and do assume that these inequalities also hold for all $y \in \R^d$. This gives (v). 

The standard arguments \cite[pages 75-76]{bib:CZ} implies that $T_t$ is symmetric, so for each fixed $t>0$ the property (iv) holds for almost all $(x,y)$ with respect to the Lebesque measure on $\Rd \times \Rd$. 

Let $f_{t,x}(y) = u(t,x,y)$. Fix $t>0$, $x_0, y_0 \in \R^d$, $r > 0$. From (iv) (for almost all $(x,y) \in \R^d \times \R^d$) and the semigroup property we have 
\begin{align*}
\int_{B(y_0,r)} u(t,x_0,y) dy = \int_{B(y_0,r)} T_{\frac{t}{2}} f_{\frac{t}{2}, x_0}(y) dy
\, .
\end{align*}
Since $f_{\frac{t}{2}, x_0} \in L^\infty(\R^d)$, (ii) gives that $T_{\frac{t}{2}} f_{\frac{t}{2}, x_0} \in C_b(\R^d)$. Therefore we may and do assume that for each fixed $t>0$ and $x \in \R^d$, $u(t,x,y)$ is continuous as a function of $y$. Fixed $t>0$. For any $x, y \in \R^d$ we have
\begin{align*}
u(t,x,y) = \int_{\R^d} \int_{\R^d} u(t/3,x,z)u(t/3,z,w)u(t/3,w,y) dw dz
\, . 
\end{align*}
For any fixed $z, w \in \R^d$, $u(t/3,z,x) \rightarrow u(t/3,z,x_0)$ and $u(t/3,w,y) \rightarrow u(t/3, w,y_0)$ as $x \rightarrow x_0$ and $y \rightarrow y_0$. By the dominated convergence theorem we get (iii). This also completes (iv) for all $x,y \in \R^d$, $t>0$. 
\end{proof} 
 
The potential operator for $(T_t)$ is defined as follows
\begin{align*}
V f(x) & = \int_0^\infty T_t f(x) dt = \ex^x \left[\int_0^\infty e_q(t) f(X_t) dt \right]
\, ,
\end{align*}
for a nonnegative Borel function $f$ on $\R^d$. 

\begin{lemma}
\label{lm:bounded}
Let $q \in L_{\loc}^{\infty}$, $q \ge 0$. If $q(x) \rightarrow \infty$ as $|x| \rightarrow \infty$, then $\left\|V \chi_{\R^d}\right\|_\infty < \infty$. 
\end{lemma}

\begin{proof}
Since $q(x) \rightarrow \infty$ as $|x| \rightarrow \infty$, there exists $R > 1$ such that $q(x) \geq 1$ for $|x| \geq R$. Denote: $A=\overline{B}(0,R)^c$, $B=B(0,2R)$. For any $0 < N < \infty$ let $f_N(x) = \ex^x\left[\int_0^N e_q(t) dt \right]$. Let $x \in B$. We have
\begin{align*}
f_N(x) & = \ex^x\left[\tau_B \geq N; \int_0^N e_q(t) dt \right] + \ex^x\left[\tau_B < N; \int_0^N e_q(t) dt \right] \\
       & = \ex^x\left[\tau_B \geq N; \int_0^N e_q(t) dt \right] + \ex^x\left[\tau_B < N; \int_0^{\tau_B} e_q(t) dt \right] 
         + \ex^x\left[\tau_B < N; \int_{\tau_B}^N e_q(t) dt \right] \\
       & \leq 2 \ex^x\left[\int_0^{\tau_B} e_q(t) dt \right] + \ex^x\left[\tau_B < N; \int_{\tau_B}^N e_q(t) dt \right] \\
       & \leq 2 \ex^x \tau_B + \ex^x\left[\tau_B < N; \int_{\tau_B}^N e_q(t) dt \right] \leq C R^\alpha 
         + \ex^x\left[\tau_B < N; \int_{\tau_B}^N e_q(t) dt \right]
\, .
\end{align*}
It is enough to estimate the last expected value. By a change of variables and the strong Markov property, we obtain
\begin{align*}
\ex^x & \left[\tau_B < N;  \int_{\tau_B}^N e_q(t) dt \right]  = \ex^x\left[\tau_B < N; e^{-\int_0^{\tau_B} q(X_s) ds }\int_{\tau_B}^N                
                                                                 e^{-\int_{\tau_B}^t q(X_s) ds }dt \right] \\
                                 & \leq \ex^x\left[\tau_B < N; \int_{\tau_B}^N e^{-\int_{\tau_B}^t q(X_s) ds }dt \right] 
                                   = \ex^x\left[\tau_B < N; \int_0^{N-\tau_B} e^{-\int_{\tau_B}^{t+\tau_B} q(X_s) ds }dt\right]\\ 
                                 & \leq \ex^x\left[\tau_B < N; \int_0^N e^{-\int_{\tau_B}^{t+\tau_B} q(X_s) ds }dt\right] 
                                   \leq \ex^x\left[\tau_B < N; \ex^{X_{\tau_B}} \left[\int_0^N e^{-\int_0^t q(X_s) ds }dt\right]\right]
\, .
\end{align*}
Thus 
\begin{align}
\label{eq:fN}
f_N(x) & \leq C_R + \ex^x f_N(X_{\tau_B})
\, , & x \in B
\, .
\end{align}

Let now $x \in B^c$. Observe that $B(x,1) \subset A$. Recalling that $q \geq 1$ on $A$, similarly as before, we have
\begin{align*}
f_N(x) & = \ex^x\left[\tau_A \geq N; \int_0^N e_q(t) dt \right] + \ex^x\left[\tau_A < N; \int_0^N e_q(t) dt \right]          \\
       & \leq 2 \ex^x\left[\int_0^{\tau_A}e^{-\int_0^tq(X_s)ds}dt\right] 
         +      \ex^x\left[\tau_A < N;e_q(\tau_A)\ex^{X_{\tau_A}}\left[\int_0^N e_q(t)dt\right]\right] \\
       & \leq 2\ex^x[\int_0^{\tau_A} e^{-t} dt] + \ex^x\left[\tau_A < N; e^{-\tau_A}\ex^{X_{\tau_A}}\left[\int_0^N e_q(t)dt\right]\right] \\
       & \leq 2 + \ex^x\left[\tau_A < N; e^{-\tau_{B(x,1)}}\ex^{X_{\tau_A}}\left[\int_0^N e_q(t)dt\right]\right] \\
       & \leq 2 + \sup_{x \in B} f_N(x) \ex^0[e^{-\tau_{B(0,1)}}]
\, .
\end{align*}
Using this and \eqref{eq:fN} we get $\sup_{x \in B} f_N(x) \leq C_R + 2 + \sup_{x \in B} f_N(x) \ex^0[e^{-\tau_{B(0,1)}}]$. Since $\ex^0[e^{-\tau_{B(0,1)}}] = C <1$, we obtain that $\sup_{x \in B} f_N(x) \leq \frac{C_R + 2}{1-C}$. Recalling that for $x \in B^c$ we have $f_N(x) \leq 2 + \sup_{x \in B} f_N(x) \ex^0[e^{-\tau_{B(0,1)}}]$, we conclude that $f_N$ is bounded everywhere and uniformly in relation to $N$, which finishes the proof. 
\end{proof}

Under the assumptions $q \in L_{\loc}^{\infty}$, $q \ge 0$, $\lim_{|x| \to \infty} q(x) = \infty$, by Lemma \ref{lm:bounded} and standard arguments \cite[Theorem 3.18]{bib:CZ}, we obtain that the operator $V$ has a symmetric kernel given by $V(x,y)=\int_0^\infty u(t,x,y) dt$, that is $V f(x)= \int_{\R^d} V(x,y)f(y)dy$.

The $q$-Green operator for an open set $D$ is defined by the formula
\begin{align*}
V_D f(x) & = \ex^x \left[\int_0^{\tau_D} e_q(t) f(X_t) dt \right]
\, ,
\end{align*}
for a nonnegative Borel function $f$ on $D$. Observe that $V_D \chi_{\R^d} (x) = v_D(x)$. Additionally, if $D^{'}$ is an open set such that $D \subset D^{'} \subseteq \R^d$ and $f$ is a nonnegative Borel function on $D^{'}$, then by the strong Markov property, we have 
\begin{equation}
\label{eq:pot1}
\begin{split}
V_{D^{'}}f(x) & = \ex^x \left[\int_0^{\tau_D} e_q(t) f(X_t) dt \right] + \ex^x \left[\int_{\tau_D}^{\tau_{D^{'}}} e_q(t) f(X_t) dt \right] \\
              & = V_D f(x) + \ex^x \left[e^{-\int_0^{\tau_D} q(X_s) ds} \int_{\tau_D}^{\tau_{D^{'}}} e^{-\int_{\tau_D}^t q(X_s) ds} f(X_t) dt \right] \\
              & = V_D f(x) + \ex^x \left[e_q(\tau_D) \ex^{X_{\tau_D}}\left[ \int_0^{\tau_{D^{'}}} e_q(t) f(X_t) dt \right]\right] \\
              & = V_D f(x) + \ex^x[e_q(\tau_D)V_{D^{'}}f(X_{\tau_D})], \ \ \ x \in D.
\end{split}
\end{equation}
We will use (\ref{eq:pot1}) to obtain the following property of $\varphi_1$. Under the assumptions $q \in L_{\loc}^{\infty}$, $q \ge 0$, $\lim_{|x| \to \infty} q(x) = \infty$ we have $T_t \varphi_1 = e^{-\lambda_1 t} \varphi_1$ which implies $\varphi_1(x) = \lambda_1 V \varphi_1(x)$. Now (\ref{eq:pot1}) applied for $f = \varphi_1$, $D' = \Rd$ and an open set $D \subset \Rd$ gives
\begin{align}
\label{eq:pot3}
\varphi_1(x) & = \lambda_1 V_D \varphi_1(x) + \ex^x[e_q(\tau_D) \varphi_1(X_{\tau_D})]
\, , & x \in D
\, .
\end{align}

If $D \subseteq \R^d$ is a regular (say bounded Lipschitz) open set, then similarly as before, the operator $V_D$ is given by symmetric kernel $V_D(x,y)$, that is, $V_D f(x) = \int_{D} V_D(x,y)f(y)dy$ (see \cite[page 58]{bib:BB1}). The function $V_D(x,y)$ is called the {\it{$q$-Green function}} of $D$ and since $q \geq 0$, it is clear that in our case $V_D(x,y) \leq G_D(x,y)$.

We say that Borel function $f$ on $\R^d$ is $q$-harmonic in an open set $D \subset \R^d$ if
\begin{align}
\label{def:harm}
f(x) & = \ex^x\left[e_q(\tau_U) f(X_{\tau_U})\right]
\, , & x \in U
\, ,
\end{align}
for every bounded open set $U$ with $\overline{U}$ contained in $D$. It is called regular $q$-harmonic in $D$ if \eqref{def:harm} holds for $U=D$. It is well known \cite{bib:BB1}, page 83, that every function regular $q$-harmonic in $D$ is $q$-harmonic in $D$. If $D$ is unbounded, then by the usual convention we understand that in (\ref{def:harm}) $\ex^x\left[e_q(\tau_D) f(X_{\tau_D})\right] = \ex^x\left[\tau_D < \infty; e_q(\tau_D) f(X_{\tau_D})\right]$.
The Borel function $f$ on $\Rd$ is said to be $q$-superharmonic in an open set $D \subset \R^d$ if 
\begin{align}
\label{def:superharm}
f(x) & \geq \ex^x\left[e_q(\tau_U) f(X_{\tau_U})\right]
\, , & x \in U
\, ,
\end{align}
for every bounded open set $U$ with $\overline{U}$ contained in $D$. We always understand that the expectation in (\ref{def:harm}) and (\ref{def:superharm}) is absolutely convergent.

For an open set $D \subset \R^d$ the {\it{gauge function}} is defined by $u_D(x) = \ex^x[e_q(\tau_D)]$, $x \in D$ (see e.g. \cite[page 58]{bib:BB1}, \cite{bib:CS1}, \cite{bib:CZ}). When it is bounded in $D$, we say that $(D,q)$ is gaugeable. If $D$ is a bounded domain with the exterior cone property, then the condition $q \geq 0$ gives that $(D,q)$ is gaugeable and for $f \geq 0$ we have
\begin{align}
\label{eq:IWF}
\ex^x[e_q(\tau_D)f(X_{\tau_D})] & = \cA \int_D V_D(x,y) \int_{D^c} \frac{f(z)}{|z-y|^{d+\alpha}} dz dy
\, , & x \in D
\, 
\end{align}
by \cite[formula (17) of Section 2 and Theorem 4.10]{bib:BB1}.

The following estimate will be very useful in the proof of Lemma \ref{lm:cruc}. It follows from \cite[Lemma 4]{bib:Kw} for $\gamma>0$; for $\gamma = 0$ it is trivial. For any $\gamma \geq 0 $, $\gamma \neq d$, 
\begin{align}
\label{eq:cruest}
 \int_{B(x,|x|/4)^c}(1+|y|)^{-\gamma}|x-y|^{-d-\alpha}dy & \le C_{\gamma}|x|^{-\gamma^{'}}
  \, , & |x| \geq 1
  \, ,
\end{align}
where $\gamma^{'}=\min(\gamma + \alpha, d + \alpha)$.

The next lemma gives an important estimate which will be needed in the proofs of Theorem \ref{th:eig} and Theorem \ref{th:char}. The proof of Lemma \ref{lm:cruc} is similar to the proof of \cite[Theorem 1]{bib:Kw}.

\begin{lemma}
\label{lm:cruc}
Let $q \in L_{loc}^{\infty}$, $q \ge 0$ and $q(x) \rightarrow \infty$ as $|x| \rightarrow \infty$. Put $D=B(x,1)$. Let $f$ be a nonnegative and bounded function on $\R^d$ such that for any $|x| \geq 3$ we have
\begin{align*}
f(x) & \le  C^{(1)}_q v_D(x)\left(\sup_{y \in B\left(x,\frac{|x|}{2}\right)}f(y) 
          + \int_{B\left(x,\frac{|x|}{2}\right)^c} f(z)|z-x|^{-d-\alpha}dz \right)
\, .
\end{align*}
Then $f(x) \le C^{(2)}_q v_D(x)|x|^{-d-\alpha}$ for all $|x| \geq 3$.
\end{lemma}

\begin{proof}
Suppose that for some $\gamma \geq 0$, $\gamma \neq d$, and any $x \in \R^d$ we have $f(x) \leq C_\gamma (1+|x|)^{-\gamma}$. It is clearly true for $\gamma=0$. Then, for $|x| \geq 3$, we have
\begin{align}
\label{eq:bound2}
f(x) & \le  C_{q,\gamma} v_D(x)\left(\sup_{y \in B\left(x,\frac{|x|}{2}\right)}f(y)
       + \int_{B\left(x,\frac{|x|}{2}\right)^c} (1+|z|)^{-\gamma} |z-x|^{-d-\alpha}dz \right) 
\, .
\end{align}
Hence, by \eqref{eq:cruest}, 
\begin{align}
\label{eq:bound4}
f(x) & \le  C_{q,\gamma} v_D(x)\left(\sup_{y \in B\left(x,\frac{|x|}{2}\right)}f(y) + |x|^{-\gamma^{'}} \right)
\, , & |x| \geq 3
\, ,
\end{align}
with $\gamma^{'} = \min(\gamma + \alpha, d + \alpha)$. Observe that $|x| \leq 2|y|$ for $y \in B\left(x,\frac{|x|}{2}\right)$. Hence
\begin{align}
\label{eq:est2}
|x|^{\gamma^{'}} f(x) & \le  C^{(1)}_{q,\gamma} v_D(x)\left(\sup_{y \in B\left(x,\frac{|x|}{2}\right)}|y|^{\gamma^{'}}f(y) + 1\right)
 \, . 
\end{align}
Denote: $g(s)=\sup_{y \in B(0,s)} |y|^{\gamma^{'} }f(y)$. It is clear that $g$ is nondecreasing and 
\begin{align}
\label{eq:bound}
g(s)\leq C^{(2)}s^{\gamma^{'}}
\, .
\end{align} 
We will show that $g(s)$ is bounded too. 

Indeed, observe that by definition of $v_D$ we have $v_D(x) \leq \min \left\{\ex^x \tau_D, (\inf_{y \in D} q(y))^{-1}\right\}$. Since $\lim_{|x| \rightarrow \infty} q(x) = \infty$, $v_D(x) \rightarrow 0$ as $|x| \rightarrow \infty$. Thus there exists $R \geq 3$ such that $C^{(1)}_{q,\gamma} v_D(x) \leq 2^{-\gamma^{'} - 1}$ for $|x| \geq R$. By \eqref{eq:est2}, for $R \leq |x| \leq s$ we get 
\begin{align*}
|x|^{\gamma^{'}} f(x) & \le 2^{-\gamma^{'} - 1} g(2|x|) +  2^{-\gamma^{'} - 1} \le 2^{-\gamma^{'} - 1} g(2s) + 2^{-\gamma^{'} - 1}
\, .
\end{align*}
On the other hand, for $|x| \leq R$ we have 
\begin{align*}
|x|^{\gamma^{'}} f(x) & \le g(R) \le  g(R) + 2^{-\gamma^{'} - 1} g(2s) 
\, .
\end{align*}
Consequently, $g(2s) \geq 2^{\gamma^{'} + 1}\left( g(s) - C^{(3)}_{q,\gamma}\right)$ when $s \geq R$. If $g(s) \ge C^{(3)}_{q,\gamma}$ then by induction, 
\begin{align}
\label{eq:bound1}
g(2^ns) & \geq 2^{n(\gamma^{'} + 1)} g(s) - C^{(3)}_{q,\gamma} \left( \frac{2^{n(\gamma^{'} + 1)} - 1}{1 - 2^{-\gamma^{'} - 1}}\right) 
\, , & n=1,2,... .
\end{align}
Suppose now that for some $s \geq R$ we have $g(s) \geq \left( 1+ \frac{1}{1 - 2^{-\gamma^{'} - 1}}\right)C^{(3)}_{q,\gamma}$. By \eqref{eq:bound} and \eqref{eq:bound1}, we get
\begin{align*}
C^{(2)} 2^{n\gamma^{'}}s^{\gamma^{'}} & \geq g(2^ns) \geq \left( 2^{n(\gamma^{'} + 1)} + \frac{1}{1 - 2^{-\gamma^{'} - 1}}\right)C^{(3)}_{q,\gamma} \geq 2^{n(\gamma^{'} + 1)} C^{(3)}_{q,\gamma}
 \, , & n=1,2,... .
\end{align*}
This gives a contradiction and $g(s)$ is bounded. Hence 
\begin{align}
\label{eq:bound3}
f(x) & \leq C_{q,\gamma} (1+|x|)^{-\gamma^{'}}
\, , & x \in \R^d
\, ,
\end{align}
where $\gamma^{'} = \min(\gamma + \alpha, d + \alpha)$. 

By \eqref{eq:bound3}, we may write the estimates \eqref{eq:bound2} with $\gamma = \gamma^{'}$ and, consequently, we get \eqref{eq:bound4} with new, larger $\gamma^{'}$. Starting from \eqref{eq:bound4}, we can repeat our reasoning and we obtain the estimate \eqref{eq:bound3} again, but now with new, larger $\gamma^{'}$.

Applying this argument repeatedly, we can improve the degree of the estimate \eqref{eq:bound3} in each next step. If after some step we get $\gamma^{'} = d$ (see \eqref{eq:cruest}), then we put $\gamma = d - \frac{\alpha}{2}$ in the next one. It is clear that after $\left\lfloor 2 + \frac{d}{\alpha}\right\rfloor$ steps we obtain that $f(x)  \leq C_q (1+|x|)^{-d-\alpha}$, $x \in \R^d$. By \eqref{eq:bound4}, this also gives $f(x) \leq C_q v_D(x)|x|^{-d-\alpha}$ for $|x| \geq 3$.
\end{proof}

\medskip

\section{Uniform estimates of $q$-harmonic functions}

In this section we obtain uniform estimates of $q$-harmonic functions in balls. "Uniform" means that the constants in these estimates do not depend on the potential $q$. In studying IU in next sections it will be crucial that these constants do not depend on $q$. The proofs of the results in this section adapt the ideas from \cite{bib:BKK} and \cite{bib:SW}, where the $\alpha$-harmonic functions were considered. 

Lemma \ref{lm:est} concerns a comparability of functions $u_D$ (the gauge function) and $v_D$ in the case of balls and plays the crucial role in the proofs of Theorem \ref{th:eig} and Theorem \ref{th:bhi}. The proof of this very important lemma is very similar to its $\alpha$-stable equivalent, which was proved in \cite{bib:SW} first time. We use the same idea with cut-off function and properties of fractional Laplacian.

\begin{lemma}
\label{lm:est}
Let $q \in L^\infty_{\loc}$, $q \geq 0$. Let $r > 0$ and $0 < \kappa <1$. There exists a constants $C_{r,\kappa}$ such that for any $x \in \R^d$ and $D = B(x,r)$
\begin{align}
\label{eq:est}
C_{r,\kappa}^{-1} \, v_D(y)
  & \le
 u_D(y)
    \le
   C_{r,\kappa} \, v_D(y)
  \, , & y \in B(x,\kappa r)
  \, .
\end{align}
\end{lemma}
\begin{proof}
Fix $0<\kappa<1$. Let $f \in C^2(\R^d)$ be a function such that $f \equiv 1$ on $B(x,\kappa r)$, $f \equiv 0$ on $B(x,r)^c$ and $0 \leq f \leq 1$. By \cite[Proposition 3.16]{bib:BB1}, we have for $z \in D$
\begin{align*}
 V_D\left(-(-\Delta)^{\frac{\alpha}{2}}f - qf\right)(z) = - f(z)
  \, .
\end{align*}
Here it is worth to point out that in \cite{bib:BB1} $e_q(t)$ is defined in a slightly different way than in our paper (namely, in \cite{bib:BB1} it is defined without a minus sign).

For $z \in B(x,\kappa r)$ it follows that
\begin{align*}
 \int_D V_D(z,y)(-\Delta)^{\frac{\alpha}{2}}f(y)dy & = f(z)- \int_D V_D(z,y)q(y)f(y)dy    \\
                                                   & \geq 1 - \int_D V_D(z,y)q(y)dy       \\
                                                   & = 1 - \ex^z\left[\int_0^{\tau_D} e_q(t)q(X_t)dt\right]         
  \, . 
\end{align*}
Noting that $\Phi(t)=q(X_t)$ is locally integrable in $(0,\infty)$ almost surely, we have that $e_q(t)$ is locally absolutely continuous in $(0,\infty)$ a.s.. Then, by the theory of Lebesgue integration (see e.g. \cite[proof of the Proposition 3.16]{bib:CZ} and \cite[formula (64), section 4]{bib:CZ}), 
\begin{align*}
\int_0^{\tau_D} e_q(t)q(X_t)dt = 1 - e_q(\tau_D)                                                 
\, .
\end{align*}                                      
Hence
\begin{align*}
u_D(z) = \ex^z[e_q(\tau_D)] \leq \int_D V_D(z,y)(-\Delta)^{\frac{\alpha}{2}}f(y)dy \leq \left\|(-\Delta)^{\frac{\alpha}{2}}f\right\|_\infty v_D(z)      \, 
\end{align*} 
for $z \in B(x,\kappa r)$. Since $f \in C^2_c(\R^d)$, we have $\left\|(-\Delta)^{\frac{\alpha}{2}}f\right\|_\infty < \infty$. On the other hand, by \eqref{eq:IWF}, for any $z \in B(x,r)$, we have
\begin{align*}
u_D(z) & = \int_D V_D(z,y) \int_{D^c} \frac{dwdy}{|w-y|^{d+\alpha}} \geq  \int_D V_D(z,y) \int_{B(x,2r)^c} \frac{dwdy}{|w-y|^{d+\alpha}} \\
       & \geq  C_r \int_D V_D(z,y)dy \int_{B(x,2r)^c} \frac{dw}{|w-x|^{d+\alpha}} = \frac{C_r}{r^\alpha} v_D(z)
\, .
\end{align*} 
\end{proof}

\begin{lemma}
Let $q \in L_{loc}^{\infty}$, $q \geq 0$, $r > 0$ and $\kappa \in (0,1)$. There exists a constant $C_{r,\kappa}$ such that if $D = B(x_0,r)$, $x_0 \in \R^d$, and $f(x) = \ex^x[e_q(\tau_D)f(X_{\tau_D})]$ for $x \in D$, $f \geq 0$, then
\begin{align}
\label{eq:func}
    f(x)
   \le &
    C_{r,\kappa} \int_{B(x_0,\kappa r)^c}\frac{f(y)}{|y-x_0|^{d+\alpha}}dy
  \, , & x \in B(x_0,\kappa r)
  \, .
\end{align}
\end{lemma}

\begin{proof}
Let $\gamma = (1+\kappa)r/2$. By definition, the function $f$ is regular $q$-harmonic in $D$. Recall that regular $q$-harmonicity implies $q$-harmonicity and the equality \eqref{def:harm} holds for $U=B(x_0,\delta) \subset D$, where $\delta \in (\gamma, r)$. Then for $\delta \in (\gamma, r)$ and any $x \in B(x_0,\kappa r)$ we have 
\begin{align*}
f(x) = \ex^x[e_q(\tau_{B(x_0,\delta)})f(X_{\tau_{B(x_0,\delta)}})] \leq \ex^x[f(X_{\tau_{B(x_0,\delta)}})]
\, .
\end{align*}
To estimate the last expectation we follow the proof of \cite[Lemma 6]{bib:BKK}. It is known (see \cite{bib:BlGR}) that for each $x \in B(x_0,\delta)$ the $\pr^x$ distribution of $X(\tau_{B(x_0,\delta)})$ has a density given by the formula
\begin{align*}
P_{x_0,\delta}(x,y) & = C_{\alpha,d} \left(\frac{\delta^2-|x-x_0|^2}{|y-x_0|^2-\delta^2}\right)^{\alpha/2} \frac{1}{|x-y|^d}
\, , &  |y-x_0|>\delta
\, ,
\end{align*}
 and  $P_{x_0,\delta}(x,y)=0$, when $|y-x_0|\le \delta$, $C_{\alpha,d} = \Gamma(d/2) \pi^{-d/2 -1} \sin(\pi \alpha/2)$. Hence, by Fubini-Tonelli theorem
\begin{align*}
f(x) & \leq \frac{1}{r-\gamma}\int_{\gamma}^r \ex^x[f(X_{\tau_{B(x_0,\delta)}})] d\delta = \int_{B^c(x_0,\gamma)}K(x,y)f(y)dy
\, , & x \in B(x_0,\kappa r)
\, ,
\end{align*}
where
\begin{align*}
K(x,y) & = \frac{1}{r-\gamma}\int_{\gamma}^{r \wedge |y-x_0|} P_{x_0,\delta}(x,y) d\delta = \frac{C_{\alpha,d}}{r-\gamma}\int_{\gamma}^{r \wedge |y-x_0|} \left(\frac{\delta^2-|x-x_0|^2}{|y-x_0|^2-\delta^2}\right)^{\alpha/2} \frac{1}{|x-y|^d} d\delta
\, ,
\end{align*}
for $y \in B^c(x_0,\gamma)$. The inequalities
\begin{align*}
\frac{|x-y|}{|y-x_0|} & \geq \frac{|y-x_0|-|x-x_0|}{|y-x_0|} \geq 1 - \frac{\kappa r}{\gamma}
\, , & \frac{|y-x_0|+\delta}{|y-x_0|} \geq 1
\,
\end{align*}
and 
\begin{align*}
\delta^2 - |x-x_0|^2 \leq r^2
\, 
\end{align*}
gives that 
\begin{align*}
K(x,y) \leq \frac{C_{\kappa,r}}{|y-x_0|^{d+\alpha/2}}\int_{\gamma}^{r \wedge |y-x_0|}\frac{d\delta}{(|y-x_0|-\delta)^{\alpha/2}} \leq C_{\kappa,r}|y-x_0|^{-d-\alpha}
\, .
\end{align*}
Hence
\begin{align*}
f(x) & \leq C_{\kappa,r} \int_{B^c(x_0,\gamma)} |y-x_0|^{-d-\alpha} f(y) dy \leq C_{\kappa,r} \int_{B^c(x_0,\kappa r)} |y-x_0|^{-d-\alpha} f(y) dy
\, , & x \in B(x_0,\kappa r)
\, ,
\end{align*}
which ends the proof.
\end{proof}

A main and crucial tool to study the intrinsic ultracontractivity for stable semigroups on unbounded open sets in \cite{bib:Kw} was the uniform boundary Harnack inequality for functions $\alpha$-harmonic in an arbitrary open set $D \subset \R^d$ with a constant independent of radius of ball including the domain of $\alpha$-harmonicity (see \cite[Lemma 3]{bib:Kw}. The idea of such strong version of this inequality comes from the papers \cite{bib:SW} and \cite{bib:BKK}, where the functions harmonic with respect to symmetric stable process were considered. In our case it suffices to prove the weaker version of such inequality only for balls. 

\begin{theorem} 
\label{th:bhi}
Let $q \in L^\infty_{\loc}$, $q \geq 0$ and $r > 0$. There exists a constant $C$ such that if $D = B(x_0,r)$, $x_0 \in \R^d$, and $f(x) = \ex^x[e_q(\tau_D)f(X_{\tau_D})]$ for $x \in D$, $f \geq 0$, then
\begin{align}
\label{eq:bhi}
  C^{-1} v_D(x) \int_{B(x_0,\frac{r}{2})^c}\frac{f(y)}{|y-x_0|^{d+\alpha}}dy
  & \le
   \, f(x)
   \le
    C v_D(x) \int_{B(x_0,\frac{r}{2})^c}\frac{f(y)}{|y-x_0|^{d+\alpha}}dy
  \, ,
\end{align}
for $x \in B(x_0,\frac{r}{2})$.  
\end{theorem}
\begin{proof}
First we prove \eqref{eq:bhi} for $r=1$. Let $x \in B(x_0, 1/2)$. Recall that the equality $f(x) = \ex^x[e_q(\tau_D)f(X_{\tau_D})]$, $x \in D$, implies \eqref{def:harm} for $U = B(x_0, 3/4) \subset D$. We have
\begin{align*}
 f(x) & = \ex^x[X_{\tau_{B(x_0, 3/4)}} \in B(x_0, 7/8)^c; e_q(\tau_{B(x_0, 3/4)})f(X_{\tau_{B(x_0, 3/4)}})] \\
      & + \ex^x[X_{\tau_{B(x_0, 3/4)}} \in B(x_0, 7/8) \backslash B(x_0, 3/4); e_q(\tau_{B(x_0, 3/4)})f(X_{\tau_{B(x_0,3/4)}})] \\ 
      & = f_1(x) + f_2(x)
  \, 
\end{align*}
Using the representation \eqref{eq:IWF} for $f_1$, we easy show that 
\begin{align*}
 C^{-1} v_{B(x_0, 3/4)}(x) & \int_{B(x_0, 7/8)^c} \frac{f(z)}{|z-x_0|^{d+\alpha}}dz  \leq f_1(x) \\
 & \leq C v_{B(x_0, 3/4)}(x) \int_{B(x_0, 7/8)^c} \frac{f(z)}{|z-x_0|^{d+\alpha}}dz
 \, , &  \, x \in B(x_0, 1/2)
\, .
\end{align*}
For $f_2$ we have 
\begin{align*}
 f_2(x) & \leq  u_{B(x_0, 3/4)}(x) \sup_{y \in B(x_0, 7/8)} f(y) \\ 
        & \leq C v_{B(x_0, 3/4)}(x) \int_{B(x_0, 7/8)^c} \frac{f(z)}{|z-x_0|^{d+\alpha}}dz
 \, , &  \, x \in B(x_0, 1/2)
\, 
\end{align*}
by \eqref{eq:est} and \eqref{eq:func}. Thus 
\begin{align*}
 C^{-1} v_{B(x_0, 3/4)}(x) & \int_{B(x_0, 7/8)^c} \frac{f(z)}{|z-x_0|^{d+\alpha}}dz  \leq f(x) \\
 & \leq C v_{B(x_0, 3/4)}(x) \int_{B(x_0, 7/8)^c} \frac{f(z)}{|z-x_0|^{d+\alpha}}dz
 \, , &  \, x \in B(x_0, 1/2)
\, .
\end{align*}
Clearly, $v_{B(x_0, 3/4)}(x) \leq v_{B(x_0, 1)}(x)$ and $\int_{B(x_0, 7/8)^c} \frac{f(z)}{|z-x_0|^{d+\alpha}}dz \leq \int_{B(x_0, 1/2)^c} \frac{f(z)}{|z-x_0|^{d+\alpha}}dz$. It suffices to show the opposite inequalities. By \eqref{eq:pot1} and \eqref{eq:est}, we have
\begin{align*}
 v_{B(x_0, 1)}(x) & \leq v_{B(x_0, 3/4)}(x) + u_{B(x_0, 3/4)}(x) \sup_{y \in B(x_0,1)} v_{B(x_0,1)}(y) \\
 &                  \leq C v_{B(x_0,3/4)}(x)
 \, , &  \, x \in B(x_0, 1/2)
\, .
\end{align*}
The last inequality follows by the fact that $v_{B(x_0,1)}(y) \leq \ex^y \tau_{B(x_0,1)} \leq C$.
Similarly, by \eqref{eq:func} we get
\begin{align*}
\int_{B(x_0, 1/2)^c} \frac{f(z)}{|z-x_0|^{d+\alpha}}dz & \leq  \int_{B(x_0, 7/8)^c} \frac{f(z)}{|z-x_0|^{d+\alpha}}dz + C \sup_{y \in B(x_0, 7/8)}f(y)\\
                                                       & \leq C \int_{B(x_0, 7/8)^c} \frac{f(z)}{|z-x_0|^{d+\alpha}}dz
\, .
\end{align*}
This completes the proof of \eqref{eq:bhi} for $r=1$. Now we prove these estimates for an arbitrary fixed $r>0$. By the scaling property (see e.g. \cite[page 265]{bib:BK}), $(X_t, \pr^x)\stackrel{d}{=} (rX_{r^{-\alpha}t}, \pr^{\frac{x}{r}})$. For an open set $U$ we have 
\begin{align*}
\tau^{rX_{r^{-\alpha}t}}_U & = \inf\left\{t>0: rX_{r^{-\alpha}t} \notin U\right\} = r^{\alpha}\inf\left\{r^{-\alpha}t>0: X_{r^{-\alpha}t} \notin r^{-1}U\right\}   \\
        & = r^{\alpha}\inf\left\{s>0: X_s \notin r^{-1}U\right\} = r^{\alpha} \tau^{X_t}_{r^{-1}U} = r^{\alpha} \tau_{r^{-1}U}
\, .
\end{align*}
We get 
\begin{align}
\label{eq:distr}
\cL_{\pr^x}(X_t,  \tau_U, X_{\tau_U}) = \cL_{\pr^{\frac{x}{r}}}(rX_{r^{-\alpha}t}, r^{\alpha} \tau_{r^{-1}U}, rX_{\tau_{r^{-1}U}})
\, ,
\end{align}
where $\cL_{\pr^x}$ denote the distribution with respect to $\pr^x$. By this, we obtain
\begin{align*}
 f(x) & = \ex^x[e_q(\tau_{B(x_0,r)})f(X_{\tau_{B(x_0,r)}})] 
        = \ex^x\left[\exp\left(-\int_0^{\tau_{B(x_0,r)}} q(X_s)ds\right) f(X_{\tau_{B(x_0,r)}})\right]     \\
      & = \ex^y\left[\exp\left(-\int_0^{r^\alpha\tau_{B(y_0,1)}} q(rX_{r^{-\alpha}s})ds\right) f(rX_{\tau_{B(y_0,1)}})\right]
\, ,
\end{align*}
where $y_0 = r^{-1}x_0$ and $y = r^{-1}x$. A simple change of variables yields that for $x \in B(x_0, r)$ we have
\begin{align*}
 f(x) = \ex^y\left[\exp\left(-\int_0^{\tau_{B(y_0,1)}} r^\alpha q(rX_s)ds\right) f(rX_{\tau_{B(y_0,1)}})\right]
      = \ex^y[e_{q_r}(\tau_{B(y_0,1)})f_r(X_{\tau_{B(y_0,1)}})]
\, ,
\end{align*}
where $f_r(z) = f(rz)$ and $q_r(z) = r^\alpha q(rz)$. It follows that for $y \in B(y_0,1)$ we have 
\begin{align*}
f_r(y) =  \ex^y[e_{q_r}(\tau_{B(y_0,1)})f_r(X_{\tau_{B(y_0,1)}})]
\, .
\end{align*}
By using the inequalities \eqref{eq:bhi} for a potential $q_r$, a function $f_r$ and $B(y_0,1)$, we obtain for $y \in B(y_0,1/2)$
\begin{equation}
\label{eq:unit}
\begin{split}
C^{-1} & \ex^y\left[\int_0^{\tau_{B(y_0,1)}} \exp\left(-\int_0^t q_r(X_s)ds\right)dt\right]\int_{B(y_0, 1/2)^c} \frac{f_r(z)}{|z-y_0|^{d+\alpha}}dz \\
& \leq f_r(y) \leq C \ex^y\left[\int_0^{\tau_{B(y_0,1)}} \exp\left(-\int_0^t q_r(X_s)ds\right)dt\right]\int_{B(y_0, 1/2)^c} \frac{f_r(z)}{|z-y_0|^{d+\alpha}}dz .
\end{split}
\end{equation}
Simple changes of variables give
\begin{align}
\label{eq:fr}
\int_{B(y_0, 1/2)^c} \frac{f_r(z)}{|z-y_0|^{d+\alpha}}dz = r^\alpha \int_{B(ry_0, r/2)^c} \frac{f(z)}{|z-ry_0|^{d+\alpha}}dz
\,
\end{align}
and 
\begin{align*}
  \int_0^{\tau_{B(y_0,1)}} \exp\left(-\int_0^t q_r(X_s)ds\right)dt 
& = \int_0^{\tau_{B(y_0,1)}} \exp\left(-\int_0^{r^\alpha t} q(rX_{r^{-\alpha}s})ds\right)dt                            \\
& = r^{-\alpha} \int_0^{r^\alpha \tau_{r^{-1}B(ry_0,r)}} \exp\left(-\int_0^t q(rX_{r^{-\alpha}s})ds\right)dt 
\, .
\end{align*}
Furthermore, by this and \eqref{eq:distr},
\begin{align*}
\ex^y \left[\int_0^{\tau_{B(y_0,1)}} \exp\left(-\int_0^t q_r(X_s)ds\right)dt\right] 
& = r^{-\alpha} \ex^y \left[\int_0^{r^\alpha \tau_{r^{-1}B(ry_0,r)}} \exp\left(-\int_0^t q(rX_{r^{-\alpha}s})ds\right)dt \right] \\
& = r^{-\alpha} \ex^{ry} \left[\int_0^{\tau_{B(ry_0,r)}} \exp\left(-\int_0^t q(X_s)ds\right)dt \right]
\, .
\end{align*}
Recalling that $x = ry$ and $x_0 = r y_0$, by \eqref{eq:unit} and \eqref{eq:fr}, we conclude that
\begin{align*}
C^{-1} \ex^x & \left[\int_0^{\tau_{B(x_0,r)}} \exp\left(-\int_0^t q(X_s)ds\right)dt\right] \int_{B(x_0, r/2)^c} \frac{f(z)}{|z-x_0|^{d+\alpha}}dz \leq f(x) \\
& \leq C  \ex^x\left[\int_0^{\tau_{B(x_0,r)}} \exp\left(-\int_0^t q(X_s)ds\right)dt\right]\int_{B(x_0, r/2)^c} \frac{f(z)}{|z-x_0|^{d+\alpha}}dz
\, 
\end{align*}
for $x \in B(x_0, r/2)$.
\end{proof}

Under the assumptions of Theorem \ref{th:bhi} we obtain the following corollary. It will be very important step in the proof of the characterization of IU.

\begin{corollary}
\label{cor:bhisc}
Let $q \in L^\infty_{\loc}$, $q \geq 0$. Assume that there is $R > 0$ such that $q(x) \geq 1$ for $|x| \geq R$. Then there exists a constant $C$ such that if $r > 0$, $x_0 \in \Rd$, $|x_0| - r \ge R$ and $f(x) = \ex^x[e_q(\tau_{B(x_0,r)})f(X_{\tau_{B(x_0,r)}})]$ for $x \in B(x_0,r)$, $f \ge 0$ then we have
\begin{align}
\label{eq:bhisc}
   f(x) &
   \le
    C  \int_{B(x_0,\frac{r}{2})^c}\frac{f(y)}{|y-x_0|^{d+\alpha}}dy
  \, , &  x \in B\left(x_0,\frac{r}{2}\right)
\, .
\end{align}
\end{corollary}
\begin{proof}
By condition $|x_0| - r \geq R$ we have that $q \geq 1$ on $B(x_0,r)$. The desired inequality is a simple consequence of \eqref{eq:bhi} and the following estimate
\begin{align*}
\ex^x \left[\int_0^{\tau_{B(x_0,r)}} \exp\left(-\int_0^t q(X_s)ds\right)dt\right] 
\leq \ex^x \left[\int_0^{\tau_{B(x_0,r)}} e^{-t}dt\right]
\leq \int_0^\infty e^{-t} dt = 1
\, .
\end{align*}
\end{proof}

\begin{lemma}
\label{lm:supharm}
Let $q \in L_{loc}^{\infty}$, $q \geq 0$. For each fixed $t>0$ the function $T_t\chi_{\R^d}(x)$ is $q$-superharmonic in every open set $D \subset \R^d$.  
\end{lemma}

\begin{proof}
Fixed $t>0$ and $D \subset \R^d$. Let $U$ be an arbitrary open bounded subset of $D$ such that $\overline{U} \subset D$. 
By a simple change of variables and the strong Markov property, we have
\begin{align*}
T_t\chi_{\R^d}(x) & = \ex^x[e_q(t)] \geq  \ex^x[e_q(t+\tau_U)] = \ex^x\left[e_q(\tau_U) e^{-\int_{\tau_U}^{t+\tau_U} q(X_s)ds}\right]              \\
                             & =    \ex^x\left[e_q(\tau_U) e^{-\int_0^t q(X_{s+\tau_U})ds}\right]                     
                               =    \ex^x\left[e_q(\tau_U) \ex^{X_{\tau_U}}[e_q(t)]\right] \\
                             & =    \ex^x\left[e_q(\tau_U) T_t\chi_{\R^d}(X_{\tau_U})\right]
\, , & x \in U
\, . 
\end{align*}
\end{proof}

\section{Compactness of $T_t$}

The following lemma gives the simple characterization of the compactness of $T_t$. 

\begin{lemma}
\label{lm:cmpct1}
Let $q \in L_{loc}^{\infty}$, $q \geq 0$ and $t>0$. Then the operator $T_t$ is compact if and only if $T_t\chi_{\R^d}(x) \rightarrow 0$ as $|x| \rightarrow \infty$.
\end{lemma}
\begin{proof}
Let $t > 0$ be fixed. We first assume that $\lim_{|x| \rightarrow \infty}T_t\chi_{\R^d}(x)=0$.  Let $(V_{r,t})$, $r >0$, be the family of operators given by kernels $v_r(t,x,y)=u(t,x,y)\chi_{B(0,r)}(y)$, that is $V_{r,t}f(x) = \int_{\R^d}v_r(t,x,y)f(y)dy$, $f \in L^2(\R^d)$. We have
\begin{align*}
\int_{\R^d} \int_{\R^d} (v_r(t,x,y))^2 dxdy & =    \int_{B(0,r)} \int_{\R^d} (u(t,x,y))^2 dxdy 
                                            \leq \int_{B(0,r)} \int_{\R^d} (p(t,y-x))^2 dxdy       \\
                                          & \leq C_t \int_{B(0,r)} \int_{\R^d} p(t,y-x) dxdy = C_t |B(0,r)| < \infty                                                       
\, .
\end{align*} 
Hence $V_{r,t}$ is the Hilbert-Schmidt operator, so it is compact. Furthermore, by the Cauchy-Schwarz inequality, we have
\begin{align*}
\left\|T_tf - V_{r,t}f \right\|_2^2 & =    \int_{\R^d} \left|\int_{B(0,r)^c}u(t,x,y)f(y)dy\right|^2 dx 
                                  \leq \int_{\R^d} \int_{B(0,r)^c}u(t,x,y)\left|f(y)\right|^2dy dx \\
                                & \leq \left\|f\right\|_2^2 \sup_{y \in B(0,r)^c} T_t\chi_{\R^d}(y)
\, .
\end{align*} 
It follows that we can aproximate $T_t$ by compact operators $V_{r,t}$ in the operator norm. Thus $T_t$ is compact.

Now we prove the opposite implication. We follow the idea from proof of \cite[Lemma 1]{bib:Kw}. Fix $t > 0$. Assume that $T_t$ is compact. Suppose that for some sequence $\left\{x_n\right\}_{n=1}^\infty$ such that $x_n \rightarrow \infty$ we have $T_t \chi_{\R^d}(x_n) \geq M >0$. Take $r >0$ large enough, so that $\int_{B(0,\frac{r}{2})^c} p(\frac{t}{2},y)dy < \frac{M}{4}$. We have
\begin{align*}
T_t \chi_{\R^d}(x_n) = \int_{\R^d} \int_{\R^d} u\left(\frac{t}{2},x_n,z\right)u\left(\frac{t}{2},z,y\right)dzdy
\, 
\end{align*} 
and
\begin{align*}
\int_{B(x_n,r)} \int_{B(x_n,\frac{r}{2})} u\left(\frac{t}{2},x_n,z\right)u\left(\frac{t}{2},z,y\right)dzdy 
             \leq C_t \int_{B(x_n,r)} T_{\frac{t}{2}}\chi_{B(x_n,r)}(z)dz
\, ,
\end{align*} 
\begin{align*}
\int_{B(x_n,r)^c} \int_{B(x_n,\frac{r}{2})} & u\left(\frac{t}{2},x_n,z\right)u\left(\frac{t}{2},z,y\right)dzdy  \\
             & \leq \int_{B(x_n,\frac{r}{2})}u\left(\frac{t}{2},x_n,z\right) \int_{B(z,\frac{r}{2})^c} p\left(\frac{t}{2},y-z\right) dydz 
               < \frac{M}{4}
\, ,
\end{align*} 
\begin{align*}
\int_{\R^d} \int_{B(x_n,\frac{r}{2})^c} u\left(\frac{t}{2},x_n,z\right)u\left(\frac{t}{2},z,y\right)dzdy 
             \leq \int_{B(x_n,\frac{r}{2})^c}p\left(\frac{t}{2},z-x_n\right) dz < \frac{M}{4}
\, .
\end{align*} 
Thus $T_t\chi_{\R^d}(x_n) < C_t \int_{B(x_n,r)} T_{\frac{t}{2}}\chi_{B(x_n,r)}(z)dz +\frac{M}{2}$. 

By taking a subsequence of $\left\{x_n\right\}$ if necessary, we may and do assume that $B(x_n,r)$ are pairwise disjoint. Hence $(\chi_{B(x_n,r)})_{n>0}$ is the orthogonal sequence of uniformly bounded functions in $L^2(\R^d)$. Moreover, the Schwarz inequality gives 
\begin{align*}
\int_{B(x_n,r)} \left(T_{\frac{t}{2}}\chi_{B(x_n,r)}(z)\right)^2 dz 
               & \geq \frac{1}{|B(x_n,r)|}\left(\int_{B(x_n,r)} T_{\frac{t}{2}}\chi_{B(x_n,r)}(z) dz\right)^2 \\
               & \geq \frac{1}{C^2_t|B(x_n,r)|}\left(T_t \chi_{\R^d}(x_n) - \frac{M}{2}\right)^2 \geq \frac{M^2}{4C^2_t|B(x_n,r)|}               
\, .  
\end{align*} 
By compactness of $T_t$, we can choose the subsequence of $T_t \chi_{B(x_n,r)}$ convergent in $L^2$-norm to some function $f \in L^2(\R^d)$. Thus for infinitely many $n$
\begin{align*}
\left(\int_{B(x_n,r)} \left(f(z)\right)^2 dz \right)^{1/2} 
               & \geq \left(\int_{B(x_n,r)} \left(T_t\chi_{B(x_n,r)}(z)\right)^2 dz \right)^{1/2}            \\
               & -    \left(\int_{B(x_n,r)} \left(f(z) - T_t\chi_{B(x_n,r)}(z)\right)^2 dz \right)^{1/2}     \\
               & \geq \frac{M}{2C_t\sqrt{|B(x_n,r)|}} - \left\|f - T_t\chi_{B(x_n,r)}\right\|_2 > \frac{M}{4C_t\sqrt{|B(x_n,r)|}}    
\, .  
\end{align*} 
This gives a contradiction. Hence $T_t\chi_{\R^d}(x) \rightarrow 0$ as $|x| \rightarrow \infty$.
\end{proof}

\begin{proof}[Proof of Lemma \ref{lm:compact}]
For any $x \in \R^d$ let us put $D=B(x,1)$. For any $t>0$ we have
\begin{align*}
 T_t\chi_{\R^d}(x) & = \ex^x[e_q(t)] =    \ex^x\left[\tau_D \geq t;e^{-\int_0^t q(X_s) ds} \right]   
                                  +    \ex^x\left[\tau_D < t;e^{-\int_0^t q(X_s) ds} \right]     \\
                  &               \leq e^{-t \inf_{y \in D} q(y)} + \ex^x\left[e^{-\int_0^{\tau_D} q(X_s) ds} \right]  
                                  \leq e^{-t \inf_{y \in D} q(y)} + \ex^0\left[e^{-\inf_{y \in D} q(y) \tau_{B(0,1)}} \right]
\, .
\end{align*} 
Since $q(x) \rightarrow \infty$ as $|x| \rightarrow \infty$, we obtain $\lim_{|x| \rightarrow \infty}T_t\chi_{\R^d}(x)=0$.
Now the assertion of the lemma follows from Lemma \ref{lm:cmpct1}.
\end{proof}

\section{Intrinsic ultracontractivity of $T_t$}

\begin{proof}[Proof of Theorem \ref{th:eig}]
We first prove the upper bound. For $|x| < 3$ and $D=B(x,1)$, by formula \eqref{eq:pot3} and estimate \eqref{eq:est}, we have 
\begin{align*}
\varphi_1(x) \leq \left\|\varphi_1\right\|_\infty (\lambda_1 v_D(x) + u_D(x) )\leq  C_q v_D(x) \leq C_q v_D(x) (1+|x|)^{-d-\alpha}
\, .
\end{align*}

Let now $|x| \ge 3$. Putting $r=\frac{|x|}{2}$ and $D=B(x,1)$, by \eqref{eq:pot3} and \eqref{eq:IWF}, we have 
\begin{align*}
 \varphi_1(x) & =    \lambda_1 \int_D V_D(x,y)\varphi_1(y)dy + \ex^x[X_{\tau_D} \in D^c \cap B(x,r); e_q(\tau_D) \varphi_1(X_{\tau_D})]  \\
              & +    \ex^x[X_{\tau_D} \in B(x,r)^c; e_q(\tau_D) \varphi_1(X_{\tau_D})]                           
                \le  \lambda_1 v_D(x) \sup_{y \in B(x,r)}\varphi_1(y)  +  u_D(x) \sup_{y \in B(x,r)}\varphi_1(y)                         \\           
              & +    \cA \int_D V_D(x,y) \int_{B(x,r)^c} \varphi_1(z) |z-y|^{-d-\alpha}dzdy            
\, .
\end{align*}
From \eqref{eq:est} we obtain
\begin{align*}
\varphi_1(x)  & \le  \lambda_1 v_D(x) \sup_{y \in B(x,r)}\varphi_1(y)  + C v_D(x) \sup_{y \in B(x,r)}\varphi_1(y)                         \\           
              & +    \cA \int_D V_D(x,y)dy \int_{B(x,r)^c} \varphi_1(z) |z-x|^{-d-\alpha}dz                                            \\
              & \le C_q v_D(x)\left(\sup_{y \in B(x,r)}\varphi_1(y) +  \int_{B(x,r)^c} \varphi_1(z) |z-x|^{-d-\alpha}dz\right)
\, .
\end{align*}
By Lemma \ref{lm:cruc} applied to $f=\varphi_1$, we get $\varphi_1(x) \leq C_q v_D(x)|x|^{-d-\alpha}$ for $|x| \geq 3$. The upper bound of the theorem is proved.

To show the lower bound we use \eqref{eq:pot3} again. Let $|x| \leq 2$ and $D=B(x,1)$. We have $\varphi_1(x) \geq \lambda_1 v_D(x) \inf_{y \in B(0,3)} \varphi_1(y) \geq C_q v_D(x) (1 + |x|)^{-d-\alpha}$.

Let now $|x| > 2$ and $D=B(x,1)$. By (\ref{eq:pot3}) and \eqref{eq:IWF} we have
\begin{align*}
\varphi_1(x) & \geq  \ex^x[e_q(\tau_D) \varphi_1(X_{\tau_D})] = C \int_D V_D(x,y) \int_{D^c} \varphi_1(z) |z-y|^{-d-\alpha}dzdy  \\
             & \geq  C \int_D V_D(x,y) \int_{B(0,1)} \varphi_1(z) |z-y|^{-d-\alpha}dzdy \geq C_q v_D(x) |x|^{-d-\alpha}
  \,  .
\end{align*}

Clearly, $v_D(x) \leq \int_{\R^d} V(x,y)dy$. By Lemma \ref{lm:bounded} $\left\|V\chi_{\R^d}\right\|_\infty < \infty$. Then, from \eqref{eq:pot1} for $D' = \Rd$, $f = \chi_{\Rd}$ and \eqref{eq:est}, we obtain $\int_{\R^d} V(x,y)dy  \leq C_q v_D(x)$.
\end{proof}

\begin{proof}[Proof of Theorem \ref{th:char}]
The condition (i) implies (ii) by definition of IU and the upper bound of Theorem \ref{th:eig}.

By (ii) we have
\begin{align*}
\ex^x[t<\tau_{\overline{B}(0,r)^c}; e_q(t)] & \leq  \ex^x[X_t \in \overline{B}(0,r)^c; e_q(t)] \\
                                          & =     \int_{\overline{B}(0,r)^c} u(t,x,y) dy \leq C_{q,t} (1+|x|)^{-d-\alpha} \leq C_{q,t} (1+r)^{-d-\alpha}
\,  ,
\end{align*}
for $x \in \overline{B}(0,r)^c$. Thus (iii) is proved. 

Now we prove the implication (iii) $\Rightarrow$ (iv). Let $R > 1$ be large enough, so that $q(x) \geq 1$ for $|x| \geq R$. Let $|x| \geq 2R$, $r=|x|/2$ and $D = B(x,r)$. By condition (iii) and the strong Markov property, we have 
\begin{equation}
\label{eq:Tt}
\begin{split}
T_t\chi_{\R^d}(x) & =    \ex^x\left[\frac{t}{2}<\tau_D; e_q(t)\right] + \ex^x\left[\frac{t}{2} \geq \tau_D; e_q(t)\right]   \\
   & \leq \ex^x\left[\frac{t}{2}<\tau_{\overline{B}(0,r)^c}; e_q\left(\frac{t}{2}\right)\right] 
     + \ex^x\left[e_q(\tau_D) \ex^{X_{\tau_D}} \left[ e_q\left(\frac{t}{2}\right)\right]\right]   \\
   & \leq C_{q,t} (1 + |x|)^{-d-\alpha} + \ex^x\left[e_q(\tau_D) T_{\frac{t}{2}}\chi_{\R^d}(X_{\tau_D})\right]
\end{split}
\end{equation}
We need to estimate the last expectation. Put
\begin{equation*}
\label{def2}
f(y) = 
\begin{cases}
\ex^y\left[e_q(\tau_D) T_{\frac{t}{2}}\chi_{\R^d}(X_{\tau_D})\right]		 &	\text{for} \quad y \in D,      \\
T_{\frac{t}{2}}\chi_{\R^d}(y)																			       &  \text{for} \quad y \in D^c .   	
\end{cases}
\end{equation*}
Then 
\begin{align*}
f(y) & = \ex^y\left[e_q(\tau_D) f(X_{\tau_D})\right]
\, , & y \in D
\, ,
\end{align*}
and from \eqref{eq:bhisc} and $q$-superharmonicity of function $T_{\frac{t}{2}}\chi_{\R^d}(y)$ (see Lemma \ref{lm:supharm}), we obtain 
\begin{align}
\label{eq:f}
f(z) & \leq C \int_{B(x,\frac{r}{2})^c}\frac{f(y)}{|y-x|^{d+\alpha}}dy  
       \leq C \int_{B(x,\frac{r}{2})^c}\frac{T_{\frac{t}{2}}\chi_{\R^d}(y)}{|y-x|^{d+\alpha}}dy
\, , & z \in B(x,r/2)
\, .
\end{align}
Consequently, by \eqref{eq:Tt} and \eqref{eq:f}, we have
\begin{align}
\label{eq:Tt1}
T_t\chi_{\R^d}(x)  \leq C_{q,t} (1 + |x|)^{-d-\alpha} + C \int_{B(x,\frac{r}{2})^c}\frac{T_{\frac{t}{2}}\chi_{\R^d}(y)}{|y-x|^{d+\alpha}}dy
\, .
\end{align}
Suppose now that for some $\gamma \geq 0$, $\gamma \neq d$, we have $T_t\chi_{\R^d}(x) \leq C_{q,t,\gamma} (1 + |x|)^{-\gamma}$ for all $x \in \R^d$, $t>0$. It is clear for $\gamma = 0$. Then, from \eqref{eq:Tt1} and \eqref{eq:cruest}, we obtain 
\begin{align}
\label{eq:Tt2}
T_t\chi_{\R^d}(x)  \leq C_{q,t} (1 + |x|)^{-d-\alpha} + C_{q,t,\gamma} \int_{B(x,\frac{r}{2})^c}(1 + |y|)^{-\gamma}|y-x|^{-d-\alpha}dy 
                   \leq C_{q,t,\gamma}(1 + |x|)^{-\gamma^{'}}
\, 
\end{align}
for $\gamma^{'} = \min(\gamma +\alpha, d+\alpha)$ and $|x| \geq 2R$. Clearly, we also have $T_t\chi_{\R^d}(x) \leq C_{q,t,\gamma}(1 + |x|)^{-\gamma^{'}}$ for $|x| \leq 2R$. 

Now, starting from \eqref{eq:Tt1} again and taking $\gamma = \gamma^{'}$ in \eqref{eq:Tt2}, we get the estimates \eqref{eq:Tt2} with new, larger $\gamma^{'}$. By using this argument recursively, we can improve the degree of estimate $T_t\chi_{\R^d}(x) \leq C_{q,t,\gamma}(1 + |x|)^{-\gamma^{'}}$. If it happens that $\gamma^{'} = d$ after some step, then we take $\gamma = d - \frac{\alpha}{2}$ in the next one. Applying this argument, after $\left\lfloor 2 + \frac{d}{\alpha}\right\rfloor$ steps we obtain that $T_t\chi_{\R^d}(x) \leq C_{q,t}(1 + |x|)^{-d-\alpha}$ for all $x \in \R^d$.

To complete the proof of the theorem we prove the implication (iv) $\Rightarrow$ (i). By the inequality
\begin{align*}
u(t,x,y) = \int_{\R^d} \int_{\R^d} u\left(\frac{t}{3},x,z\right)u\left(\frac{t}{3},z,v\right)u\left(\frac{t}{3},v,y\right)dvdz 
           \leq C_t T_{\frac{t}{3}}\chi_{\R^d}(x) T_{\frac{t}{3}}\chi_{\R^d}(y)
\,  ,
\end{align*}
it suffices to show that $T_t\chi_{\R^d}(x) \leq C_{q,t} \varphi_1(x)$ for $x \in \R^d$ and $t > 0$.

Let $|x|>3$, $D=B(x,1)$ and $r = \frac{|x|}{2}$. We have
\begin{align}
\label{eq:sem}
T_t\chi_{\R^d}(x) = \ex^x\left[\frac{t}{2}<\tau_D; e_q(t)\right] + \ex^x\left[\frac{t}{2} \geq \tau_D; e_q(t)\right]
\, .
\end{align}
We start by estimating the first expected value in \eqref{eq:sem}. By the Markov property, we have
\begin{align*}
\ex^x\left[\frac{t}{2}<\tau_D; e_q(t)\right] & \leq \ex^x\left[\frac{t}{2}<\tau_D; e_q\left(\frac{t}{2}\right)\right] \\ 
& = \ex^x\left[\frac{t}{4} < \tau_D; e_q\left(\frac{t}{4}\right)\ex^{X\left(\frac{t}{4}\right)}\left[\frac{t}{4}< \tau_D; e_q\left(\frac{t}{4} \right)\right]\right] \\
& \leq 
\ex^x\left[\frac{t}{4} < \tau_D; e_q\left(\frac{t}{4}\right)\right] \sup_{y \in D}T_{\frac{t}{4}}\chi_{\R^d}(y)
\, .
\end{align*}
Observing that
\begin{align*}
v_D(x) & = \ex^x\left[\int_0^{\tau_D} \frac{dv}{\exp\left(\int_0^v q(X_s)ds\right)}\right] 
           \geq \ex^x\left[\frac{t}{4} < \tau_D;\int_0^{\frac{t}{4}} \frac{dv}{\exp\left(\int_0^v q(X_s)ds\right)}\right]       \\
       &   \geq \ex^x\left[\frac{t}{4} < \tau_D; \frac{\frac{t}{4}}{\exp\left(\int_0^{\frac{t}{4}} q(X_s)ds\right)}\right]
           = \frac{t}{4} \ex^x\left[\frac{t}{4} < \tau_D; e_q\left(\frac{t}{4}\right)\right]
\, 
\end{align*}
by condition (iv) of Theorem \ref{th:char}, we obtain
\begin{align*}
\ex^x\left[\frac{t}{4} < \tau_D; e_q\left(\frac{t}{4}\right)\right] \sup_{y \in D}T_{\frac{t}{4}}\chi_{\R^d}(y)
\leq C_{q,t} v_D(x) (1 + |x|)^{-d-\alpha}
\, .
\end{align*}
Consequently, 
\begin{align}
\label{eq:sem1}
\ex^x\left[\frac{t}{2}<\tau_D; e_q(t)\right] \leq C_{q,t} v_D(x) (1 + |x|)^{-d-\alpha}
\, .
\end{align}
 
Now we find the upper bound for the second expected value in \eqref{eq:sem}. The strong Markov property, \eqref{eq:IWF}, \eqref{eq:est}, condition (iv) and \eqref{eq:cruest} yield
\begin{align*}
& \ex^x \left[\frac{t}{2} \geq \tau_D; e_q(t)\right] \leq \ex^x \left[e_q(\tau_D)\ex^{X_{\tau_D}} \left[e_q\left(\frac{t}{2}\right)\right]\right]  \\
      & = \ex^x \left[X_{\tau_D} \in B(x,r); e_q(\tau_D)\ex^{X_{\tau_D}} \left[e_q\left(\frac{t}{2}\right)\right]\right]
      + \ex^x \left[X_{\tau_D} \in B(x,r)^c; e_q(\tau_D)\ex^{X_{\tau_D}} \left[e_q\left(\frac{t}{2}\right)\right]\right]                           \\
      & \leq  u_D(x) \sup_{y \in B(x,r)} T_{t/2}\chi_{\R^d}(y) + C \int_D V_D(x,y) \int_{B(x,r)^c} T_{t/2}\chi_{\R^d}(z) |z-y|^{-d-\alpha}dzdy    \\
      & \leq C_{q,t}v_D(x)(1 + |x|)^{-d-\alpha} + C_{q,t} v_D(x) \int_{B(x,r)^c} (1 + |z|)^{-d-\alpha} |z-x|^{-d-\alpha}dz                       \\
      & \leq C_{q,t}v_D(x)(1 + |x|)^{-d-\alpha}      
\, .
\end{align*}
By \eqref{eq:sem} and \eqref{eq:sem1}, this gives  $T_t\chi_{\R^d}(x) \leq C_{q,t} v_D(x) (1 + |x|)^{-d-\alpha}$ for $|x| > 3$.

For $|x| \leq 3$ let $D = B(x,1)$. By \eqref{eq:sem} and \eqref{eq:est} we have
\begin{eqnarray*}
T_t\chi_{\R^d}(x) &\le& E^x\left[\frac{t}{2} < \tau_D; \frac{1}{t/2} \int_{0}^{t/2} e_q(s) \, ds \right] + E^x\left[\frac{t}{2} \ge \tau_D; e_q(t)\right] \\
&\le& E^x\left[\frac{t}{2} < \tau_D; \frac{1}{t/2} \int_{0}^{\tau_D} e_q(s) \, ds \right] + E^x\left[ e_q(\tau_D)\right] \\
&\le& C_t v_D(x) + u_D(x) \le C_{q,t} v_D(x) \le C_{q,t} v_D(x) (1 + |x|)^{-d - \alpha}.
\end{eqnarray*}
Finally, by Theorem \ref{th:eig}, we have $T_t\chi_{\R^d}(x) \leq C_{q,t} \varphi_1(x)$. 
\end{proof}

\begin{proof}[Proof of Theorem \ref{th:suff}]
Since $\lim_{|x| \rightarrow \infty}\frac{q(x)}{\log |x|} = \infty$, we have $\lim_{|x| \rightarrow \infty}q(x) = \infty$. Hence, by Lemma \ref{lm:compact}  each $T_t$ is compact. Moreover, we observe that
\begin{align*}
\ex^x[t<\tau_{\overline{B}(0,r)^c}; e_q(t)] & \leq  \exp(- \lambda(r) t )
\, , &  \, x \in \overline{B}(0,r)^c, r>0
\, ,
\end{align*}
where $\lambda(r)= \inf_{y \in \overline{B}(0,r)^c} q(y)$. By Theorem \ref{th:char} it is enough to show that $\exp(- \lambda(r) t ) \leq C (1 + r)^{-d-\alpha}$. But, by the assumption, there is $R>0$ such that $\lambda(r) \geq \frac{d+\alpha}{t} \log(1+r)$ for $r>R$. Thus the desired inequality holds for $r>R$. When $r \leq R$, then simply $\exp(- \lambda(r) t ) \leq 1 = (1+R)^{d+\alpha} (1+R)^{-d-\alpha} \leq C (1+r)^{-d-\alpha}$.
\end{proof}

\begin{proof}[Proof of Theorem \ref{th:nec}]
Set $r = \frac{|x|}{2}$ for $|x| \geq 2$ and $D=B(x,\epsilon)$ for an arbitrary $0< \epsilon \leq 1$. By condition (iii) of Theorem \ref{th:char} we have for $|x| \geq 2$, $t > 0$,
\begin{align*}
\pr^x(t<\tau_D) e^{-\sup_{y \in D} q(y) t} \leq \ex^x[t<\tau_D; e_q(t)] \leq \ex^x[t<\tau_{\overline{B}(0,r)^c}; e_q(t)]  \leq  C_{q,t} (1 + r)^{-d-\alpha}
\, .
\end{align*}
Hence for $0< t \leq 1$ and $|x| \geq 2$
\begin{align*}
\pr^0(1<\tau_{B(0,\epsilon)}) e^{-\sup_{y \in D} q(y) t} \leq C_{q,t} |x|^{-d-\alpha}
\, .
\end{align*}
It follows that
\begin{align*}
e^{-\sup_{y \in D} q(y) t} \leq C_{q,t,\epsilon} |x|^{-d-\alpha}
\, 
\end{align*}
and, consequently,
\begin{align*}
\frac{\sup_{y \in D} q(y)}{\log|x|} \geq \frac{\alpha +d}{t} - \frac{C_{q,t,\epsilon}}{t \log|x|} 
\, .
\end{align*}
We conclude that $\liminf_{|x| \rightarrow \infty} \frac{\sup_{y \in D} q(y)}{\log|x|} \geq \frac{\alpha +d}{t}$ for any $0 < t \leq 1$. 
\end{proof}

\section{Potentials comparable on unit balls}

\begin{proof}[Proof of Theorem \ref{th:eig1}]
We fix $x \in \R^d$. Let $M=M_{q,x}$ and $D=B(x,1)$. We have
\begin{align*}
v_D(x) & = \ex^x \left[ \int_0^{\tau_D} \exp \left( -\int_0^t q(X_s) ds \right) dt\right]  
           \geq \ex^x \left[ \int_0^{\tau_D} \exp \left( -M (1 + q(x))t  \right) dt\right] \\
       & = \frac{\ex^x \left[ 1 - \exp \left( -M (1 + q(x))\tau_D  \right) \right]}{M(1 + q(x))} 
           \geq \frac{\ex^x \left[\tau_D \geq 1; 1 - \exp \left( -M (1 + q(x))\tau_D  \right) \right]}{M(1 + q(x))} \\
       & \geq \ex^x \left[\tau_D \geq 1; 1 - e^{-M} \right](M(1 + q(x)))^{-1} \geq  2^{-1}\pr^0(\tau_{B(0,1)} \geq 1)M^{-1} (1 + q(x))^{-1}
\, .
\end{align*}
On the other hand,
\begin{align*}
v_D(x) & = \ex^x \left[ \int_0^{\tau_D} \exp \left( -\int_0^t q(X_s) ds \right) dt\right] 
           \leq \ex^x \left[ \int_0^{\tau_D} \exp \left( -M^{-1} (1 + q(x))t  \right) dt\right] \\
       & = \ex^x \left[ 1 - \exp \left( -M^{-1}(1 + q(x))\tau_D  \right) \right]M(1 + q(x))^{-1} 
           \leq \frac{M}{1 + q(x)}                     
\, .
\end{align*}
The estimate \eqref{eq:eig1} is a simple consequence of Theorem \ref{th:eig} and the above inequalities.
\end{proof} 

\begin{proof}[Justification of Example \ref{ex:ex2}]
Let $x_n, y_n \in \R^d$ be sequences such that $|x_n|=n-1+2r_n$ and $|y_n|=n-1-3r_n$, $|x_n - y_n|=5r_n$. Denote: $D_n=B(x_n,1)$, $D_n^{'}=B(x_n,r_n)$, $B_n=B(y_n,2r_n)$, $B_n^{'}=B(y_n,r_n)$. Recall that $r_n = \frac{1}{a_n^{1/\alpha}}$, $a_1 > 2^\alpha$. Hence $D_n^{'} \subset D_n$. Let $n$ be large enough, so that $B_n \subset D_n$ (see Figure 1).

\begin{figure}[t!]
\centering
\includegraphics[width=11cm]{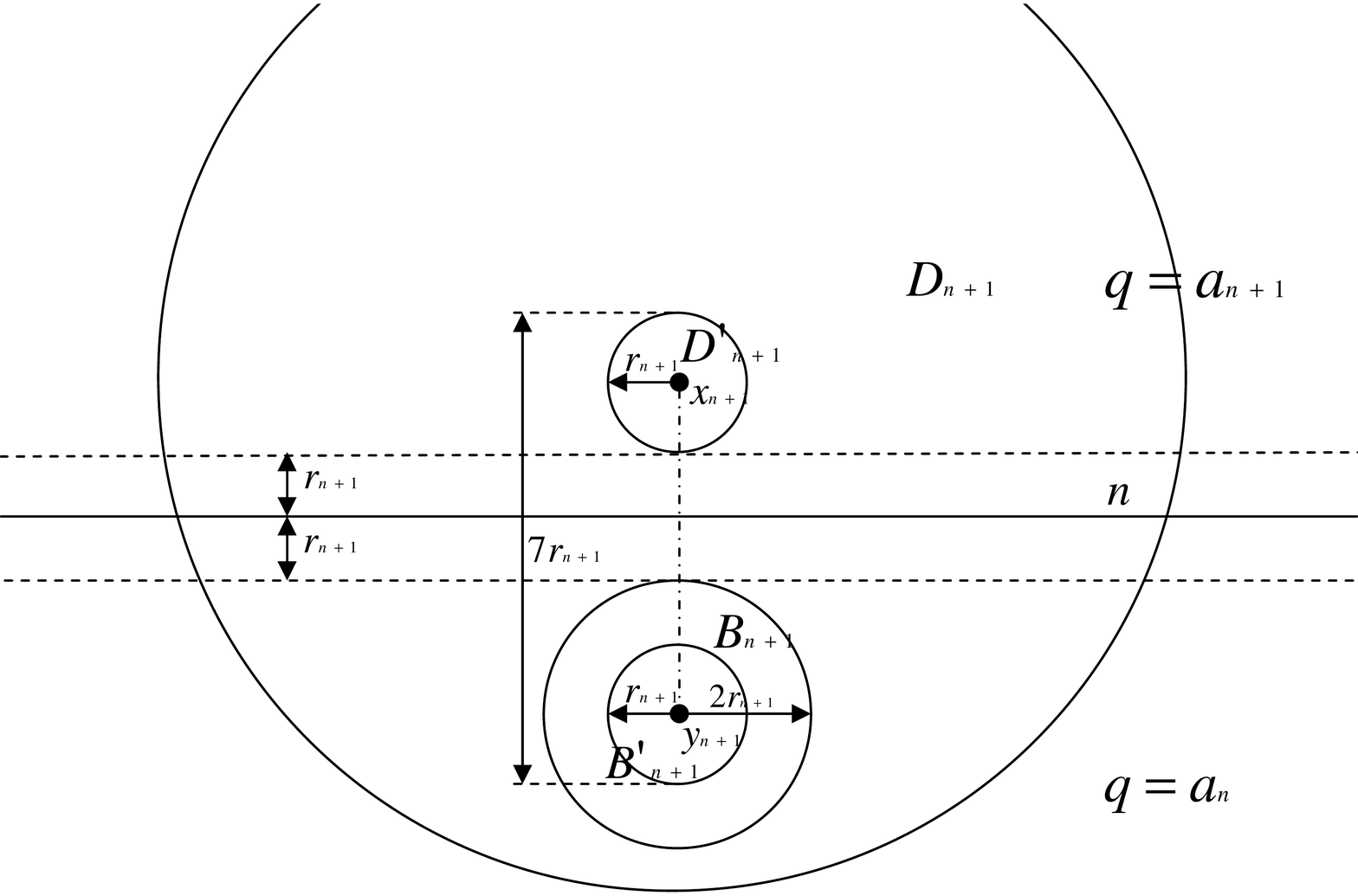}
\caption{Illustration of notation in the justification of Example \ref{ex:ex2} for $d=2$}
\end{figure}

By \eqref{eq:pot1} and \eqref{eq:IWF}, we have
\begin{align*}
v_{D_{n+1}}(x_{n+1}) & \geq \ex^{x_{n+1}} \left[ e_q\left(\tau_{D_{n+1}^{'}}\right) v_{D_{n+1}}\left(X_{\tau_{D_{n+1}^{'}}}\right)\right]  \\
       &  \geq \ex^{x_{n+1}} \left[X_{\tau_{D_{n+1}^{'}}} \in B_{n+1}^{'}; e_q\left(\tau_{D_{n+1}^{'}}\right)    
               v_{D_{n+1}}\left(X_{\tau_{D_{n+1}^{'}}}\right)\right]                                                                               \\
       & =    \cA \int_{D_{n+1}^{'}} V_{D_{n+1}^{'}}(x_{n+1},y) \int_{B_{n+1}^{'}} \frac{v_{D_{n+1}}(z)}{|y-z|^{d+\alpha}}dzdy                        \\
       & \geq \cA \int_{D_{n+1}^{'}} V_{D_{n+1}^{'}}(x_{n+1},y) \int_{B_{n+1}^{'}} \frac{v_{B_{n+1}}(z)}{|y-z|^{d+\alpha}}dzdy                        \\
       & \geq \cA\inf_{z \in B_{n+1}^{'}} v_{B_{n+1}}(z) \frac{Cr_{n+1}^d}{(7r_{n+1})^{d+\alpha}} \int_{D_{n+1}^{'}} V_{D_{n+1}^{'}}(x_{n+1},y) dy   \\          & =     C (r_{n+1})^{-\alpha} \inf_{z \in B_{n+1}^{'}} v_{B_{n+1}}(z)  v_{D_{n+1}^{'}}(x_{n+1})                           
\, .
\end{align*}

It is enough to estimate $\inf_{z \in B_{n+1}^{'}} v_{B_{n+1}}(z)$ and $v_{D_{n+1}^{'}}(x_{n+1})$. By \eqref{eq:surv} we obtain for $x \in B_{n+1}^{'}$
\begin{align*}
v_{B_{n+1}} & (x) = \ex^{x} \left[\int_0^{\tau_{B_{n+1}}}e^{-a_n t} dt \right] = \frac{\ex^{x}\left[1-e^{-a_n\tau_{B_{n+1}}}\right]}{a_n}
                 \geq \frac{\ex^{x}\left[\tau_{B_{n+1}} > a_n^{-1}; 1-e^{-a_n\tau_{B_{n+1}}}\right]}{a_n}    \\
               & \geq \frac{\ex^{x}\left[\tau_{B_{n+1}} > a_n^{-1}; 1-e^{-1}\right]}{a_n} 
                 = \frac{C\pr^x(\tau_{B_{n+1}} > a_n^{-1})}{a_n}                            
                 \geq \frac{C}{a_n} \frac{\delta_{B_{n+1}}^{\alpha/2}(x)}{a_n^{-1/2}} \geq C \frac{r_{n+1}^{\alpha/2}}{a_n^{1/2}}
                 =\frac{C}{\sqrt{a_n a_{n+1}}}
\, .
\end{align*}
By the same argument, we have $v_{D_{n+1}^{'}}(x_{n+1}) \geq C a_{n+1}^{-1}$. Hence
\begin{align*}
v_{D_{n+1}}(x_{n+1}) & \geq C (a_n a_{n+1})^{-1/2} (r_{n+1})^{-\alpha} a_{n+1}^{-1} = C (a_n a_{n+1})^{-1/2}                  
\, .
\end{align*}
It follows that $q(x_{n+1})v_{D_{n+1}}(x_{n+1}) = a_{n+1} v_{D_{n+1}}(x_{n+1}) \geq C \sqrt{\frac{a_{n+1}}{a_n}} \rightarrow \infty$ as $n \rightarrow \infty$. Due to \eqref{eq:eig0} the upper bound estimate in \eqref{eq:eig2} does not hold.
\end{proof}

\medskip

\begin{flushleft}
\textbf{Acknowledgements.} The authors thank M. Kwa\'snicki for many discussions on intrinsic ultracontractivity.
\end{flushleft}

\end{document}